%% file: PoissonScreening5.tex
\documentclass[abstracton, paper=a4, fontsize=11pt, DIV=12]{scrartcl}%

\pdfoutput=1

\usepackage{amsmath}
\usepackage{amsfonts}
\usepackage{amssymb}
\usepackage{graphicx}
\usepackage[english]{babel}
\usepackage[T1]{fontenc}
\usepackage{lmodern}
\usepackage{verbatim}

\usepackage{amsmath,amsthm,amssymb,amsfonts}
\usepackage{mathrsfs}
\usepackage[utf8]{inputenc}
\usepackage[shortlabels]{enumitem}
\usepackage{caption}
\usepackage{mathtools}
\usepackage{esint}
\usepackage{textcomp}
\usepackage[backend=bibtex,style=alphabetic,firstinits]{biblatex}
\usepackage{csquotes}
\usepackage{nicefrac}

\bibliography{bibliography.bib}

\newcommand{\dist}{\operatorname{dist}}
\newcommand{\dv}{\operatorname{div}}

\newcommand{\IR}{\mathbb{R}}
\newcommand{\IN}{\mathbb{N}}
\newcommand{\IZ}{\mathbb{Z}}

\newcommand{\dd}{\, d}

\newcommand{\supp}{\operatorname{supp}}
\newcommand{\loc}{{\mathrm{loc}}}
\newcommand{\eps}{\varepsilon}

\newcommand{\Hdot}{\dot{H}^1}

\providecommand{\U}[1]{\protect\rule{.1in}{.1in}}

\newtheorem{theorem}{Theorem}[section]

\newtheorem{condition}[theorem]{Condition}

\newtheorem{corollary}[theorem]{Corollary}

\newtheorem{definition}[theorem]{Definition}

\newtheorem{lemma}[theorem]{Lemma}
\newtheorem{notation}[theorem]{Notation}

\newtheorem{proposition}[theorem]{Proposition}
\newtheorem{remark}[theorem]{Remark}

\makeatletter
\renewenvironment{proof}[1][\proofname]{%
  \par\pushQED{\qed}\normalfont%
  \topsep6\p@\@plus6\p@\relax
  \trivlist\item[\hskip\labelsep\bfseries#1\@addpunct{.}]%
  \ignorespaces
}{%
  \popQED\endtrivlist\@endpefalse
}
\makeatother

\title{The Method of Reflections, Homogenization and Screening for Poisson
and Stokes equations in perforated domains}

\author{Richard H\"{o}fer\thanks{University Bonn}}

\mathtoolsset{showonlyrefs}
\begin{document}
\bigskip

\bigskip

\bigskip

\begin{center}
{\Large The Method of Reflections, Homogenization and Screening for Poisson
and Stokes Equations in Perforated Domains}

\bigskip

Richard H\"{o}fer\footnote{Institute for Applied Mathematics, University of Bonn, Endenicher Allee 60, 53115 Bonn, Germany}, Juan J. L. Vel\'{a}zquez\footnotemark[1]

\bigskip

\today
\end{center}

\begin{abstract}
\noindent We study the convergence of the Method of Reflections for the Dirichlet problem of the Poisson and the Stokes equations in perforated domains which consist in the exterior of balls.
We prove that the method converges if the balls are contained in a bounded region and the density of the electrostatic capacity of the balls is sufficiently small. If the capacity density is too large or the balls extend to the whole space, the method diverges,
but we provide a suitable modification of the method that converges to the solution of the Dirichlet problem also in this case.
We give new proofs of classical homogenization results for the Dirichlet problem of the Poisson and the Stokes equations in perforated domains using the (modified) Method of Reflections.
\end{abstract}

\section{Introduction}

In this paper, we consider Poisson and Stokes equations in perforated domains
\begin{equation}
-\Delta u=f\ \ \text{in\ }\mathbb{R}^{3}\diagdown K\ \ ,\ \ u=0\ \ \text{in\ }%
K, \label{S1E1}%
\end{equation}%
and
\begin{equation}
-\Delta v+\nabla p=f\ ,\ \ \nabla\cdot v=0\ \ \text{in\ }\mathbb{R}%
^{3}\diagdown K\ \ ,\ \ v=0\ \ \text{in\ }K. \label{S1E2}%
\end{equation}
where $u$ is a scalar function, and $v$ is a vector field with values in
$\mathbb{R}^{3}.$
Here, the set $K$ consists of mutually disjoint balls,
\begin{equation}
K=\bigcup_{i \in I  }\overline{B_{r_i}\left(
x_i\right)  }, \label{S1E3}%
\end{equation}
where $I$ is a finite or countable index set.

Problems analogous to \eqref{S1E1} and \eqref{S1E2} have been often studied in
the physics literature using the so-called Method of Reflections. This method
allows to obtain some formal series for the solutions of these equations which
eventually should approximate them.

However, the series obtained by means of the Method of Reflections are
divergent for problems like \eqref{S1E1} and \eqref{S1E2} where $K$ extends to
the whole space. This divergence takes place even if the source term $f$ is
compactly supported or decays very fast at infinity.

The purpose of this paper is to obtain what is the precise mathematical
meaning of the formal series obtained by means of the Method of Reflections
and to explain how these series can be used to obtain the asymptotic behaviour
of the solutions of \eqref{S1E1} and \eqref{S1E2}
in the limit of small balls and the number of balls per unit volume tending to infinity.

\subsection{The Method of Reflections}

The Method of Reflections in hydrodynamic equations was introduced by
Smoluchowski (cf. \cite{Smo11}). This method allows to approximate the solutions of
boundary value problems for the Poisson or Stokes equations in domains with
complex boundaries consisting of many connected components. We write any of
those equations as
\begin{equation}
\mathcal{L}\phi=f\ \ \text{in\ \ }\Omega\, \label{S1E4}%
\end{equation}
where $\phi$ is the solution to be computed and $f$ is a suitable source term,
and where $\mathcal{L}$ could be in principle any linear elliptic operator. We
will assume by definiteness that we wish to solve these equations in the
domain $\Omega=\mathbb{R}^{d}\diagdown\bigcup_{j}C_{j},$ where the sets
$C_{j},$ which from now on will be denoted as particles, are compact sets and
$C_{j}\cap C_{k}=\emptyset$ if $j\neq k.\ $The boundary conditions might be
Dirichlet, Neumann or Robin or any other type as long as they are linear. We
will write the boundary condition at each set $C_{j}$\ as
\begin{equation}
\mathcal{B}\phi=g_{j}\ \ \text{on\ }\partial C_{j}. \label{S1E5}%
\end{equation}

Suppose that the exterior boundary value problem outside each of the sets
$C_{j}$ can be solved explicitly, i.e., we have explicit formulas (typically in
terms of integrals) for the problems%
\begin{equation}
\mathcal{L}\psi_{j}=0\ \ \text{in\ \ }\mathbb{R}^{d}\diagdown C_{j}%
,\ \ \mathcal{B}\psi_{j}=h_{j}\text{ on }\partial C_{j}. \label{S1E6}%
\end{equation}

It is then possible to compute iteratively a solution for the boundary value
problem \eqref{S1E4}, \eqref{S1E5} in $\Omega$ as follows. We write as zero
order approximation $\Phi_{0}$ to the solution of \eqref{S1E4}, \eqref{S1E5}
just as the solution of
\begin{equation}
\mathcal{L}\Phi_{0}=f\ \text{\ in }\mathbb{R}^{d}.\ \ \ \ \  \label{S1E6a}%
\end{equation}
This solution cannot be expected to satisfy the boundary condition
\eqref{S1E5}. We then define a first order approximation to $\phi$ adding to
$\Phi_{0}$ the solutions of the problems \eqref{S1E6} where $h_{j}$ is chosen
as the difference between the desired boundary condition and the one given by
$\Phi_{0}.$ More precisely we define $\Phi_{1,j}$ as the solution of%
\begin{equation}
\mathcal{L}\Phi_{1,j}=0\ \ \text{in\ \ }\mathbb{R}^{d}\diagdown C_{j}%
,\ \ \mathcal{B}\Phi_{1,j}=g_{j}-\mathcal{B}\Phi_{0}\text{ on }\partial C_{j}.
\label{S1E7}%
\end{equation}

We then define $\Phi_{1}=\sum_{j}\Phi_{1,j}.$ Then $\Phi_{0}+\Phi_{1}$ yields
a new approximation to $\phi.$ This new approximation does not satisfy 
the boundary conditions on $\bigcup_{j}\partial C_{j}.$ We can then define a
new correction $\Phi_{2},$ defining functions $\Phi_{2,j}$ in a manner
analogous to \eqref{S1E7}. More precisely we define inductively functions
$\Phi_{k,j}$ as%
\begin{align}
\mathcal{L}\Phi_{k,j}  &  =0\ \ \text{in\ \ }\mathbb{R}^{d}\diagdown
C_{j},\ \ \mathcal{B}\Phi_{k,j}=-\mathcal{B}\left(  \sum_{\ell\neq j}%
\Phi_{k-1,\ell}\right)  \ \text{on }\partial C_{j}\ \text{for\ }%
k=2,3,...,\label{S1E7a}\\
\Phi_{k}  &  =\sum_{j}\Phi_{k,j} \label{S1E7b}%
\end{align}
\ \ 

Iterating the method, we obtain a series $\Psi_{N}=\Phi_{0}+\Phi_{1}+\Phi
_{2}+...+\Phi_{N}$. The reason, why this sequence can be hoped to converge to the solution
of the boundary value problem \eqref{S1E4}, \eqref{S1E5} is that $\Psi_N$ satisfies \eqref{S1E4} and,
by induction,
\[
	\mathcal{B} \Psi_{N} = g_j - \mathcal{B} \Phi_{N+1,j}    \ \text{on }\partial C_{j}.
\]

There are several variations of the Method of Reflections in the literature.
In some cases the corrections $\Phi_{k,j}$ are not computed simultaneously for
all the particles but sequentially in each of the particles (cf. for instance
\cite{Luke}). Nevertheless the main idea of the method is always the same, and
it consists in adding recursively the corrective terms required to have the
desired boundary conditions.

Variations of the Method of Reflection have been used extensively to compute solutions of
Poisson and Stokes equations (cf. \cite{HapBre}, \cite{IB01}, \cite{TT97}, and \cite{Kirp} to mention only a few).
However the mathematical results yielding rigorous conditions on the
convergence of this method and its precise range of applicability are much
more limited. Convergence of the Method of Reflections has been considered in
\cite{Luke} for a particular type of boundary conditions which arise naturally
in sedimentation problems, and in \cite{Tra} for Dirichlet boundary conditions.

There are several clear difficulties that one encounters when trying to prove
the
convergence of the method described above to the solution $\phi.$ If there are
infinitely many particles $C_{j}$, it is not clear whether the functions $\Phi_{k}$
would be defined since they are given by a series with infinitely many
terms. Actually, the divergence of these series might be expected in this
situation because the solutions of Poisson and Stokes equations yield long
range interactions which decay as power laws with a too slow decay. Even if
the functions $\Phi_{k}$ are well defined, the convergence of $\left\{
\Psi_{N}\right\}  _{N}$ as $N\rightarrow\infty$ is not clear. Divergence of
this series might happen if the particles $C_{j}$ are too close and their
mutual interactions do not tend to zero sufficiently fast. More
precisely, divergence is expected if
\[
\Big | \sum_{l\neq j}\Phi_{k,l}(x_{j}) \Big | >\left\vert \Phi
_{k,j}(x_{j})\right\vert
\]
for most of the particles $j$. Indeed this condition implies, that adding
$\Phi_{k}$ does not bring the function closer to the right boundary conditions
at those particles $j$.

In order to investigate the application of the Method of Reflections to Problem
\eqref{S1E1}, let us consider for simplicity the special case
of particles with equal radii distributed on a lattice, i.e.,
\begin{equation}
	\label{eq:lattice}
K=\bigcup_{x \in (d \IZ)^3 }\overline{B_{r}\left(
x\right)  } ,
\end{equation}
where $d>0$ is the particle distance.
For the analysis of the convergence, it turns out
that some characteristic length is of great importance, namely the screening length.
This concept was introduced in the physics literature in \cite{MarqRoss}. A
precise mathematical discussion of this length and its relevance in phase
transition problems driven by diffusive effects can be found in \cite{NO},
\cite{NV}. The following precise definition of the screening length is well
suited in the context of the Methods of Reflections. We consider equal charges
on all particles that are contained in a ball of radius $\rho$. Then, we look
at the potential at the particle which is at the center of this ball. This
potential is the sum of the potential that is induced by the charge on that
particular particle and the potential due to all the other particles. Then,
the screening length is the critical radius $\rho$ at which those two portions
are equal. More precisely, we define $u_j$ to be the unique solution with $u_j(x) \to 0$
as $|x| \to \infty$ of the problem
\begin{align*}
-\Delta u_{j}  &  =0\ \ \text{in \ }\mathbb{R}^{3}\diagdown B_{j},\\
u_{j}  &  =1\ \ \ \text{on\ \ }\partial B_{j}.%
\end{align*}
Then, the screening length is defined as
\begin{equation}
\Lambda:=\sup\Bigg\{  \rho>0:\sup_{\partial B_{j}}\Bigg(  \sum_{l\neq
j,\ x_{l}\in B_{\rho}}u_{l}\left(  x\right)  \Bigg)  <1\Bigg\}.  \label{S3E3}%
\end{equation}

If we now apply the Method of Reflections for Poisson equation to the system
containing only the particles in a cloud of radius $R$, i.e., for
$K_R = K \cup B_R(0)$ with $K$ as in \eqref{eq:lattice}, a sufficient condition for
convergence would be
\[
R<\Lambda.
\]
Indeed, adding $\Phi_{k}$ would then really bring the function closer to the
right boundary conditions for most of the particles, leading to the estimate
\[
\Vert\Phi_{k+1}\Vert\leq\theta\Vert\Phi_{k}\Vert
\]
in a suitable norm, where
\begin{equation}
\label{condConvergence}
\theta:=\sup_{\partial B_{k}}\left(  \sum_{j\neq k,\ x_{j}\in B_{R}}%
u_{k}\left(  x\right)  \right)  <1.
\end{equation}

This condition is similar to the sufficient condition obtained in \cite{Tra}
for the convergence of the Method of Reflections for the Laplace equation in
exterior domains with Dirichlet boundary conditions. The condition there reads
\begin{equation}
\label{cond:Traytak}
\max_{i}\sum_{k\neq i}\frac{\frac{r}{\left\vert x_{i}-x_{k}\right\vert }%
}{1-\frac{r}{\left\vert x_{i}-x_{k}\right\vert }}<1.
\end{equation}
In many particles systems with small radii and typical distance between
particles $d$, the Conditions \eqref{condConvergence} and \eqref{cond:Traytak} are roughly equivalent to%
\[
d^{-3}r\max_{i}\sum_{k\neq i}\frac{d^{3}}{\left\vert x_{i}-x_{k}\right\vert
}<C%
\]
with $C$ of order one. Approximating the sum by an integral and
assuming that the particles are contained in ball with radius $R$, this would
be equivalent to%
\[
d^{-3}r\int_{B_{R}\left(  0\right)  }\frac{dy}{\left\vert y\right\vert
}<C.%
\]
Thus, the screening length $\Lambda$ is of order $\sqrt{r^{-1}d^{3}}$.

Therefore, for general particle distributions, 
it is natural to define 
\begin{equation}
	\label{eq:capacityBound}
	 \mu_0 := \inf_{i \in I} r_i d_i^{-3},
\end{equation}
where $d_i$ denotes the distance of the particle $i$ to the closest other particle.
We will give the precise conditions for the particle distributions at the beginning of Section 2. 

\subsection{Main Results for the Screened Poisson Equation}

In order to avoid divergences but still allow for infinitely many particles,
instead of Poisson equation, we will consider first a modified version of the
problem \eqref{S1E1}, namely the screened Poisson equation
\begin{equation}
-\Delta u+\xi^{-2}u=f\ \ \text{in\ }\mathbb{R}^{3}\diagdown
K\ \ ,\ \ u=0\ \ \text{in\ }K\ \label{S1E8}%
\end{equation}
for some $\xi>0.$ The basic difference between \eqref{S1E1} and \eqref{S1E8}
is that the Green's function associated to the second problem decreases
exponentially in distances of order $\xi.$ Thus, the series defining the
functions $\Phi_{k}$ are well defined. Moreover, particle interactions decay
exponentially on distances of order $\xi$. Therefore, the series in the Method
of Reflections converges for infinitely many particles provided $\mu_0 < C \xi^{-2}
$, as stated in the following theorem.

\begin{theorem}  
\label{SeriesPoissonScreened}Suppose that $\mathcal{L}=-\Delta+\xi^{-2}$ and
$\mathcal{B}=I.$ There exists $C_0>0$ depending only on
$\alpha$ and $\kappa$ from Condition \ref{cond:particlesNotToClose} and \ref{cond:radiusSmallerXi},
and  there exists $\eps <1 $ depending only on $\kappa$ with the following properties. 
Let $\Omega=\mathbb{R}^{3}\diagdown K$ with $K$ as in
\eqref{S1E3}, and let $g_{j}=0$ on $\partial C_{j}.$ Suppose also that $f\in
H^{-1}\left(  \mathbb{R}^{3}\right)  $ and define $\Phi_{0}$ as in
\eqref{S1E6a} and inductively the functions $\Phi_{k}$ by means of
\eqref{S1E7a}, \eqref{S1E7b}. Suppose that $\mu_0$  defined in \eqref{eq:capacityBound} satisfies
\begin{equation}
\mu_0 < C_0 \xi^{-2}. \label{S1E9}%
\end{equation}

Then the series $\sum_{k=0}^{\infty}\Phi_{k}$ converges to the unique solution $u$ of
\eqref{S1E8} in $H^{1}\left(  \mathbb{R}^{3} \backslash K\right) $. Moreover,
\begin{equation}
	\label{eq:expConv}
	\Big\| \sum_{k=0}^{N}\Phi_{k} - u \Big\|_{H^1(\IR^3 \backslash K)} \leq C \eps^N \| f \|_{H^{-1}(\IR^3)},
\end{equation}
where $C$ depends only on $\xi$
In particular, the convergence is
uniform in all particle configurations satisfying \eqref{S1E9}.
\end{theorem}

As indicated above, if $ \mu_0 \gtrsim\xi^{-2}$, the condition \eqref{S1E9}
fails. In that case the series $\sum_{k=0}^{\infty}\Phi_{k}$ is in general
divergent and the Method of Reflections cannot be applied, at least not in the form
stated in Theorem \ref{SeriesPoissonScreened}. However, it turns out that it
is possible to give a meaning to the formal series arising in the Method of Reflections
in order to obtain a modified series which converges to the solution of
\eqref{S1E8}.

\begin{theorem}
\label{CapOrderOne} Suppose that Conditions \ref{cond:particlesNotToClose} and \ref{cond:radiusSmallerXi} hold with 
some constants $\alpha$ and $\kappa$, and suppose
\[
	\mu_0 \leq C_\ast \xi^{-2} ,
\]  
for some constant $C_\ast < \infty$. Then, 
there exists a double sequence $q\left(  k,N\right)  $ defined for $k, N\in\mathbb{N}$ and $0\leq
k\leq N,$ depending only on $\alpha$, $\kappa$, and $C_\ast$
with the following properties. For all $ k \in \IN$,  $\lim_{N \rightarrow \infty} q\left(  k,N\right)  =1$ ,
and for all $ f \in  H^{-1}(\IR^3)$, the sequence
\[
\Psi_{N}=\sum_{k=0}^{N}q\left(  k,N\right)  \Phi_{k}%
\]
converges as $N\rightarrow\infty$ to the unique solution $u$ of
\eqref{S1E8} in $H^{1}\left(  \mathbb{R}^{3} \backslash K\right) $. 

Moreover, there exists a constant $\eps <1$ depending only on
$\alpha$, $\kappa$,  and $C_\ast$ such that
\[
	\| \Psi_N - u \|_{H^1(\IR^3 \backslash K)} \leq C \eps^N \| f \|_{H^{-1}(\IR^3)},
\]
where $C$ depends only on $\xi$.
\end{theorem}

\subsection{The Summation Procedure and the Main Result for the Poisson Equation}

Theorem \ref{CapOrderOne} can be thought as a summation method for the
original series $\sum_{k=0}^{\infty}\Phi_{k}.$ The precise construction of the
sequence $q\left(  k,N\right)  $ will be given in Section 2.

Theorems \ref{SeriesPoissonScreened} and \ref{CapOrderOne} refer to the
Dirichlet problem for the screened Poisson equation \eqref{S1E8} containing a
parameter $\xi$ which restricts the range of interaction between particles to
the finite value $\xi.$ It is natural to ask if the result can be generalized
to the Dirichlet Problem for the Poisson equation \eqref{S1E1} which
corresponds to $\xi=\infty.$

In this case, the series \eqref{S1E7b} defining the functions $\Phi_{k}$ does
not converge if the particles extend to the whole space $\IR^3$
and then the Method of Reflections as formulated in Theorem
\ref{SeriesPoissonScreened} becomes meaningless (see also \eqref{S1E9}).

Nevertheless,  using the formal
series $\sum_{k=0}^{\infty}\Phi_{k}$, it is possible to construct an alternative series which converges to the solution of
\eqref{S1E1}.
However, the relation between the original (divergent) series and the modified one, is much more
involved than in the case of the screened Poisson equation Theorem \ref{CapOrderOne}.
Therefore, we will first give an idea of the summation method.

The summation method is based on an interpretation of the Method of Reflections
using an abstract idea of Functional Analysis in
Hilbert spaces. It is well known that by means of convenient choices of
Hilbert spaces $H$, the solutions of many boundary value problems for a large
class of equations with the form \eqref{S1E4} is equivalent to the orthogonal
projection of $\mathcal{L}^{-1}f$ to the subspace of the Hilbert space for
which the boundary conditions hold. We denote here by $\mathcal{L}^{-1}$ the
operator solving \eqref{S1E4} in the whole space, which can be easily computed
using the Green's function associated to \eqref{S1E4}. We will denote this
orthogonal projection operator providing the solution of the boundary value
problem \eqref{S1E4} by $P$. This projection maps the Hilbert space $H$
into the subspace satisfying the boundary conditions, which will be denoted by
$V.$ On the other hand, we can associate another
orthogonal projection operator $P_{j}$ to the solution of the boundary
value problem for a single particle $j$. This
projection maps $H$ in a subspace $V_{j}$ for which the boundary conditions
are satisfied at the particle $j.$ We have $ V=\cap_{j}V_{j}.$ Let  $Q_{j}$ 
denote the orthogonal projection from $H$ in the orthogonal of $V_{j}$ in $H.$

It turns out that the partial sums for the Method of Projections $\sum
_{k=0}^{N}\Phi_{k}$ can be written as%
\[
\bigg(  1-\sum_{j}Q_{j}\bigg)  ^{N}\mathcal{L}^{-1}f.
\]
Thus, the Method of Projections converges to the solution of \eqref{S1E4} if%
\begin{equation}
P=\lim_{N\rightarrow\infty}\bigg(1 -\sum_{j}Q_{j}\bigg)  ^{N}
\label{S3E4}%
\end{equation}
in some suitable way. This result would hold trivially if the subspaces
$\left\{  V_{j}\right\}  $ were mutually orthogonal. However, if the angles
between some of these subspaces are too small a geometrical argument shows
that \eqref{S3E4} will fail. 
 It is precisely condition \eqref{S1E9} that ensures
that the convergence \eqref{S3E4} takes place for the Dirichlet Problem of the screened Poisson equation
\eqref{S1E8}. This is the main idea in the
Proof of Theorem \ref{SeriesPoissonScreened}.

A related geometrical interpretation of the Method of Reflections has been
analyzed in \cite{Luke}. The method used in \cite{Luke} can be
applied to systems with finitely many particles, and the convergence of the
Method of Reflections used there, which does not treat all the particles
simultaneously but sequentially, leads to showing that
\[
\lim_{N\rightarrow\infty}\bigg(  \prod_{j}P_{j}\bigg)  ^{N}=P,%
\]
where the product is taken over the finite number of particles chosen in any
order. Actually, the Method of Reflections used in \cite{Luke} cannot be
applied in the case of Dirichlet boundary conditions. Instead, it is applied to the
Stokes system imposing the set of mixed boundary conditions at the particles
satisfied by sedimenting inertialess particles, and to the Poisson
equation with analogous boundary conditions. 

As indicated above, the convergence stated in \eqref{S3E4} cannot be expected
if \eqref{S1E9} fails. However, a geometrical argument shows that, as long as
the sum $\sum_{j}Q_{j}$ is convergent, the following convergence takes place.%
\begin{equation}
P=\lim_{N\rightarrow\infty}\bigg(  I-\gamma\sum_{j}Q_{j}\bigg)  ^{N},
\label{S3E5}%
\end{equation}
if $\gamma>0$ is small enough. Actually the right hand side can be written as a series directly related
to the original series $\sum_{k=0}^{N}\Phi_{k}$ stated in
Theorem \ref{CapOrderOne}.

For the Poisson equation with particles extending to the whole space, the series $\sum_{j}Q_{j}$ is
in general divergent. However, a similar idea can be applied by including in $\gamma$ an
additional dependence on the particle position.

\begin{theorem}
\label{ConvWholeSpace}
Let $f\in \dot{H}^{-1}\left(  \mathbb{R}^{3}\right)  .$ There exists a $\gamma_0 >0$
depending only on $\mu_0$ from Equation \eqref{eq:capacityBound} and $\kappa$ from Condition \ref{cond:particlesNotToClose} such that the sequence
\begin{equation}
\lim_{N\rightarrow\infty}\bigg(  1-\gamma\sum_{j}e^{-\left\vert
x_{j}\right\vert }Q_{j}\bigg)  ^{N} (-\Delta)^{-1} f\ \label{S3E6}%
\end{equation}
converges to the solution of \eqref{S1E1} in $\dot{H}^{1}\left(
\mathbb{R}^{3}\right)  $ for all $\gamma < \gamma_0$.
\end{theorem}

\begin{remark}
We denote by $\Hdot(\IR^3) := \{  v \in L^6(\IR^3) \colon \nabla v \in L^2 (\IR^3) \}$ the homogeneous Sobolev space
 and by
$\dot{H}^{-1}\left(  \mathbb{R}^{3}\right)  $ its dual space.
\end{remark}

\subsection{Homogenization Results}
To illustrate the possible use of the Method of Reflections,
we will give a proof of classical homogenization results in perforated domains 
using only the tools developed in this paper.

For simplicity we will only consider regular particle configurations \eqref{eq:lattice}.
We have already explained the importance of the quantity $d^{-3}r$ when we
introduced the screening length $\Lambda$. Furthermore, we can draw the
following analogy to the theory of electrostatics. The electrostatic capacity
of a conductor is the charge induced on it by a difference of potential. In
the case of the system under consideration, we consider the difference of $u$
between the surface of the sphere and sufficiently far from it, at distances
of the order of the particle distances. It turns out that $d^{-3}r$ is of the
order of the density of the electrostatic capacity of the particles of the
system. We recall that the electrostatic capacity of a sphere of radius $r$ is
$4\pi r$ (cf. \cite{Jackson}).
The role of the electrostatic capacity in the solution of the Dirichlet
problem for the Laplace equation in perforated domains was already recognized
in \cite{CiMu}, \cite{MK}. The question considered in this paper was the
homogenization problem%
\begin{equation}
-\Delta u_{r}=f\ \ \text{in\ }\Omega\diagdown K_{r}\ \ ,\ \ u_{r}%
=0\ \ \text{in\ }K_{r}\cup\partial\Omega,\ \ \label{S2E2}%
\end{equation}
where $\Omega$ is an open bounded subset of $\mathbb{R}^{n}$ and $K_{r}$ is
the sequence of domains%
\begin{equation}
K_{r}=\bigcup_{x\in\left(  d\mathbb{Z}^{3}\right)  }\overline{B_{r}\left(
x\right)  }, \label{S2E6}%
\end{equation}
where the density of electrostatic capacity $\mu=\frac{4\pi r}{d^{3}}$ is
assumed to be constant. It was proved in \cite{CiMu} that for $f\in L^{2}\left(  \Omega\right) $
the sequence of solutions $u_{r}$ converges weakly in
$H^{1}\left(  \Omega \right)  $ as $r\rightarrow0$ to the solution of
\begin{equation}
-\Delta u+\mu u=f\ \ \text{in\ }\Omega\ \ ,\ \ u=0\ \ \text{in\ }%
\partial\Omega. \label{S2E3}%
\end{equation}

The results of \cite{CiMu} do not require to assume that $\mu$ is constant,
and more general particle configurations than the ones in \eqref{S1E3} can be
considered. Generalizations have been developed, including more general elliptic operators, in particular
Stokes equations \cite{All}, \cite{DGR}, \cite{MK}. Most of the homogenization results for elliptic problems have
been obtained in bounded domains. The homogenization problem associated to
\eqref{S2E2} has been considered in \cite{NV1}, \cite{NV}. In particular, it
was proved in those papers that assuming that $f\in L^{\infty}\left(
\mathbb{R}^{n}\right)  $, the unique bounded solutions of \eqref{S2E2} converge
weakly in $H_{loc}^{1}\left(  \mathbb{R}^{n}\right)  $ as $r\rightarrow0$ to
the solution of \eqref{S2E3} with $\Omega=\mathbb{R}^{n}$. The proof of the
homogenization results in \cite{NV1}, \cite{NV} relies heavily in the
derivation of the so-called screening estimate, which states that the
fundamental solution for the Laplace equation in a perforated domain with
homogeneous Dirichlet boundary conditions decreases exponentially over
distances of the order of the screening length $\Lambda=\frac{1}{\sqrt{\mu}}.$
The proof of this estimate given in \cite{NV} uses the maximum principle for
second order elliptic operators and therefore the proof cannot be easily
generalized to higher order operators.

Since the convergence result in Theorem
\ref{ConvWholeSpace} is
uniform in particle configurations as in \eqref{S2E6}
if the capacity density remains bounded, it turns out that
it is possible to use it to derive also homogenization results not using
Maximum Principle arguments.

\begin{theorem}
\label{HomogLambZero} Suppose that $f\in H^{-1}\left(  \mathbb{R}^{3}\right)
.$ Then, the problems \eqref{S1E8} with
$K=K_{r}$ as in \eqref{S2E6} and constant $\mu = \frac{4\pi r}{d^{3}}$ have unique solutions $u_{r} \in H^{1}\left(  \mathbb{R}^{3}\right) $. In the limit $ r \to 0$, $u_r$ 
converges weakly in $H^{1}\left(  \mathbb{R}^{3}\right)$
to the unique  solution $u\in H^{1}\left(  \mathbb{R}^{3}\right) $ of the problem
\begin{equation}
-\Delta u+\mu u=f\ \ \text{in\ }\mathbb{R}^{3}.
\end{equation}
\end{theorem}

An analogous result can also be proved for the solutions of the equation
\eqref{S1E8} with $K=K_{r}$ and $r\rightarrow0$. In that case, the limit equation reads
\[
	-\Delta u+(\xi^2+\mu) u=f.
\]

The previous results can be obtained also for Stokes equation. In this case we
need to precise the meaning of solving the equations \eqref{S1E6a},
\eqref{S1E7}. We will use the standard procedure of solving the equations in
the space of divergence free functions using the pressure as a suitable
Lagrange multiplier. We will say that $\phi$ is a solution of the equation
$\mathcal{L}_{Stokes}\left(  \phi\right)  =f$ in $U$\ with $\phi=0$ in
$\partial U$ with $f\in \dot{H}^{-1}\left(  U;\IR^3\right)  $ and $U$ an open set of
$\mathbb{R}^{3}$ if $\phi\in \dot{H}^{1}\left(  U;\IR^3\right)  ,$ we have
$\nabla\cdot\phi=0$, and there exists $p\in L^{2}\left(  U\right)  $ such that
$\phi$ is a weak solution of%
\begin{equation}
-\Delta\phi+\nabla p=f\ ,\ \nabla\cdot\Phi=0\ \label{S2E8}%
\end{equation}
in the domain $U.$

\begin{theorem}
\label{StokesConvergence}Let $f\in \dot{H}^{-1}\left(  \mathbb{R}^{3};\IR^3\right)
.$ There exists a $\gamma_0 >0$ depending only on $\mu_0$ from Equation \eqref{eq:capacityBound} and $\kappa$ from Condition \ref{cond:particlesNotToClose} such that the sequence
\begin{equation}
\lim_{N\rightarrow\infty}\bigg(  1-\gamma\sum_{j}e^{-\left\vert
x_{j}\right\vert }Q_{j}\bigg)  ^{N} \mathcal{L}_{Stokes}^{-1} f\ %
\end{equation}
converges to the solution of \eqref{S1E2} in $\dot{H}^{1}\left(
\mathbb{R}^{3};\IR^3\right)  $ for all $\gamma < \gamma_0$.
\end{theorem}

Using Theorem \ref{StokesConvergence}, we can also also homogenization results for the Stokes equation.

\begin{theorem}
\label{HomogStokes} Suppose that $f\in H^{-1}\left(  \mathbb{R}^{3};\IR^3\right)  .$
Then, the problems \eqref{S1E2} with
$K=K_{r}$ as in \eqref{S2E6} and constant $\mu = \frac{6\pi r}{d^{3}}$ have unique solutions $u_{r} \in H^{1}\left(  \mathbb{R}^{3} ; \IR^3\right) $. In the limit $ r \to 0$, $u_r$  converges weakly in $
H^{1}\left(  \mathbb{R}^{3} ;\IR^3\right) $\ as $r\rightarrow0$
to the unique solution  $u \in H^{1}\left(  \mathbb{R}^{3} ;\IR^3\right)$ of
\begin{equation}
-\Delta u+\nabla p+\mu u=f\ \ \text{in\ }\mathbb{R}^{3},\ \ \nabla\cdot u=0\ \label{S2E9}.
\end{equation}
\end{theorem}

\bigskip

Related results have been obtained in \cite{All},
\cite{DGR}, \cite{MK}.
The system of equations \eqref{S2E9} is known as Brinkman equations, which is a
well established model in the theory of filtration.
It provides an
interpolation between the Stokes equation and Darcy's law in porous media (see \cite{SP} and \cite{All2}). 
All the results
in those papers have been obtained in bounded domains. Theorem \ref{HomogStokes} above
provides a new proof of this type of homogenization results by means of the
Method of Reflections. Note that the homogenization result in Theorem
\ref{HomogStokes} is valid for particle distributions in the whole space.
However, we do not think that the Method of Reflections is really needed to
prove homogenization results in unbounded domains, because seemingly the
methods of \cite{DGR} might be easily adapted to prove Theorem
\ref{HomogStokes}. We just want to emphasize that the convergence result in
Theorem \ref{StokesConvergence} is strong enough to allow the derivation of
the homogenization limit.

\subsection{Plan of the Paper}

The rest of the paper is organized as follows.

In Section 2, we will prove Theorem \ref{SeriesPoissonScreened} and \ref{CapOrderOne}.
To do so, after repeating a basic lemma from Functional Analysis, we will give the precise formulation of the Method of Reflections in terms of orthogonal projections in 
Section 2.2., which will directly lead to necessary and sufficient conditions for convergence of the series obtained by the Method of Reflections. In Section 2.3, we will provide the necessary estimate to prove Theorem \ref{SeriesPoissonScreened}.
In Section 2.4, we will explain in detail the geometrical idea leading to the summation method yielding Theorem \ref{CapOrderOne}.
In Section 2.5, we will analyze the summation method on the level of the original series obtained by the Method of Reflections.

In Section 3, we will explain the modification needed to adapt the method derived in Section 2 to the Poisson equation.
This modification basically consists in a spatial cutoff in order to solve the problem of divergent series due to the long range 
structure of the Poisson equation. This leads to the proof of Theorem \ref{ConvWholeSpace}.

In Section 4, we prove the homogenization result, Theorem \ref{HomogLambZero}. In Section 4.1, we show that Problem \eqref{S1E1} with
$K$ as in \eqref{eq:lattice} is well posed in $H^1(\IR^3)$ due to the existence of a Poincaré inequality in 
$H^1_0(\IR^3 \backslash K)$. Thereafter we give a formal derivation of the homogenization result based on the original formal series obtained by the Method of Reflections. Finally, we give the rigorous proof of Theorem \ref{HomogLambZero} using
the tools and results from the previous sections.

In Section 5, we apply the method to the Stokes equations \eqref{S1E2} in order to prove Theorem \ref{StokesConvergence} and
\ref{HomogStokes}. 
Since most parts work exactly the same way as for the Poisson equation, we refrain from going through all the details again,
but rather point out the necessary modifications.

\input{ScreenedPoisson2.tex}

\input{PoissonReflection.tex}

\input{Homogenization.tex}

\input{Stokes.tex}

\input{Conclusion.tex}

\section*{Acknowledgement}

The authors acknowledge support through the CRC 1060, the mathematics of emergent effects, of the University of Bonn,
that is funded through the German Science Foundation (DFG).

\printbibliography

\end{document}

%% file: ScreenedPoisson2.tex
\section{The Screened Poisson Equation}
\label{sec:ScreenedPoisson}

We will now specify the particle distributions that we consider throughout Section \ref{sec:ScreenedPoisson}
and Section \ref{sec:Poisson}. In Section \ref{sec:Homogenization}, we only consider special configurations.

For a  finite or countable index set  $I$ we denote  $(x_i)_{i \in I}$ and $(r_i)_{i \in I}$
the positions and radii of the particles. We denote the space that the particles occupy by 
\[
	K := \bigcup_{i \in I} B_i,
\]
where we abbreviate  $B_i = B_{r_i}(x_i)$. We only consider spherical particles, but everything also works
if we instead assume that the $i$-th particle is contained in $B_i$.

For each particle $i \in I$ we define the distance to the nearest other particle
\[
	d_i := \inf_{j \neq i} |x_i - x_j|.
\]
Then the sets $B_{\frac{d_i}{2}}(x_i)$ are disjoint.

In the following, we will always assume that the following two conditions are satisfied.

\begin{condition}
	\label{cond:Capacity} There exists a constant $\mu_0$ such that 
	\[
		r_i d_i^{-3} \leq \mu_0 \quad \text{for all} ~ i \in I.
	\]
\end{condition}

\begin{condition}
	\label{cond:particlesNotToClose}
	\item There exists a constant $\kappa > 1$ such that 
	\[
		\frac{d_i}{2} > \kappa r_i \quad \text{for all} ~ i \in I.
	\]
\end{condition}
\begin{remark}
	Without loss of generality, we will always assume $\kappa \leq 2$ in the following.
\end{remark}

The second condition is not very restrictive. First, it is satisfied for any finite number of non-touching particles. Second,
it is also satisfied for infinitely many particles, if all the radii are sufficiently small
and Condition \ref{cond:Capacity} holds.
Condition \ref{cond:Capacity} can is as an upper bound for the capacity density of the particles.

In this Section, we will additionally impose the following condition, which is only important when considering the screened
Poisson equation \eqref{S1E8} and trivially satisfied for sufficiently small particles.

\begin{condition}
	 \label{cond:radiusSmallerXi} There exists a constant $ \alpha $ such that 
	\[
		r_i \leq \alpha \xi \quad \text{for all} ~ i \in I.
	\]
\end{condition}

\subsection{Preliminaries of Functional Analysis.}

In the following, $G_0 := (-\Delta + \xi^{-2})^{-1}$ will denote the solution operator for the screened Poisson equation
in the whole space $\IR^3$. 
Then, $ G_0 f = W_\xi \ast f$, where
\begin{equation}
	\label{eq:FundamentalScreened}
	W_\xi(x) = \frac{e^{-\frac{|x|}{\xi}}}{4 \pi |x|}.
\end{equation}
Moreover, $G_0$ is an isometric isomorphism from $ H^{-1}(\IR^3) $ to $ H^1(\IR^3) $ if we modify the
standard scalar product in $ H^1(\IR^3) $ according to
\[
	(u,v)_{H^1_\xi} := (\nabla u,\nabla v)_{L^2} + \xi^{-2} (u,v)_{L^2}.
\]
We will always consider $H^1(\IR^3)$ endowed with this scalar product.

Furthermore, we will denote the dual pairing between $H^{-1}$ and $H^1$ by $\langle \cdot , \cdot \rangle$.

Moreover, we will use the following notation that differs slightly
from the usual terminology. Given any closed set $K\subset\mathbb{R}^{3}$ we
will denote as $H_{0}^{1}\left(  \mathbb{R}^{3}\diagdown K\right)  $ the
closure in the $H^{1}\left(  \mathbb{R}^{3}\right)  $ topology of the set of
functions $u\in C_{c}^{\infty}\left(  \mathbb{R}^{3}\right)  $ such that $u=0$
in $K.$ Notice that with this convention the elements of $H_{0}^{1}\left(
\mathbb{R}^{3}\diagdown K\right)  $ are also elements of $H^{1}\left(
\mathbb{R}^{3}\right)  .$

We now recall a classical Functional Analysis result which allows to interpret
the solutions of the Dirichlet problem for elliptic equations using
projections. These projection operators will be an essential tool in the rest
of this paper. 

\begin{lemma}
	\label{lem:scrpoibdry}
	Let $\Omega \subset \IR^3$ be open. 
	Then, for every $ f \in  H^{-1}(\IR^3) $,	the problem
	\begin{equation}
	\begin{aligned}
		\label{eq:screenedPoissonBall}
		-\Delta u + \xi^{-2} u &= f \quad \text{in} ~ \IR^3 \backslash \overline{\Omega}, \\
		u &= 0 \quad \text{in} ~ \overline{\Omega}
	\end{aligned}
	\end{equation}
	has a unique weak solution $ u \in H^1(\IR^3) $.
	Moreover, the solution for Problem \eqref{eq:screenedPoissonBall},
	is given by
	\begin{equation}
		P_\Omega G_0 f,
	\end{equation}
	where $ P_\Omega $ is the orthogonal projection from $ H^1 (\IR^3) $ to the subspace
	$ H^1_0(\IR^3 \backslash \overline{\Omega})$. 
\end{lemma}

\begin{proof}
	Existence and uniqueness follow directly from the Riesz Representation Theorem since the weak formulation reads
	\[
		(u,v)_{H^1(\IR^3)} = \langle v, f \rangle \qquad \text{for all} \quad v \in H^1_0(\IR^3 \backslash \overline{B_i}).
	\]
	Furthermore, denoting by $u$ the solution to Problem \eqref{eq:screenedPoissonBall}, we have for $ v \in H^1_0(\IR^3 \backslash \overline{B_i}) $
	\[
		( G_0 f - u, v)_{H^1(\IR^3)} = \langle v,f\rangle - \langle v,f\rangle = 0.
	\] 
	Hence, $u = P_{i} G_0 $.
\end{proof}

\subsection{Formulation of the Method of Reflections Using Orthogonal Projections}

We now recall the Method of Reflections  and give directly an interpretation involving the projection operators mentioned in the introduction. 
These projection operators are defined by 
\begin{equation}
	\label{eq:Q}
	Q_{i} = 1 - P_{i},
\end{equation}
where $P_{i} := P_{B_i}$ are the projection operators from Lemma \ref{lem:scrpoibdry}. Thus, $Q_{i}$ is the orthogonal projection in $H^1(\IR^3)$ to the subspace
 $H^1_0(\IR^3 \backslash \overline{B_i})^\perp$. Equivalently, for $u \in H^1(\IR^3)$, $Q_{i} u$ solves
\begin{align}
	\label{eq:characterizationOfQ}
	-\Delta Q_{i} u + \xi^{-2} Q_{i} u &= 0 \quad \text{in} ~ \IR^3 \backslash \overline{B_i}, \\
		Q_{i} u &=  u \quad \text{in} ~ \overline{B_i}.
\end{align} 
This also yields the characterization
\begin{equation}
\label{eq:characterizationOfOrthogonal}
	H_0^1(\IR^3 \backslash \overline{B_i})^\perp = 
	\{ v \in H^1(\IR^3) \colon -\Delta v + \xi^{-2} v = 0 \text{ in } \IR^3 \backslash \overline{B_i}\}.
\end{equation}
Here and in the following, we write "$ f = 0$  in $\Omega$" for some $f \in H^{-1}(\IR^3)$ if
$f$ is supported in $\IR^3 \backslash \Omega$. 
 
For $f \in H^{-1}(\IR^3)$, we define $\Phi_0 := G_0 f$. Then, the first order correction for a particle $i$
is given by
$\Phi_{1,i} := -Q_{i} \Phi_0$, and the first order approximation for the solution is obtained by subtracting from $\Phi_0$ the correctors $\Phi_{1,i}$ for all the particles, i.e.,
\[
	\Psi_1 = \Phi_0 + \sum_{i\in I} \Phi_{1,i}.
\]
Similarly, the $k$-th order correction for a particle $i$ is given by
\begin{equation}
	\Phi_{k,i} = - Q_{i} \sum_{j \neq i} \Phi_{k-1,j}.
\end{equation}
Then, we define
\begin{equation}
	\label{eq:kthOrderCorrection}
	\Phi_k = \sum_i \Phi_{k,i}.
\end{equation}
and the $k$-th order approximation $\Psi_k = \Phi_0 + \dots + \Phi_k$.
Therefore, the Method of Reflections yields the series 
\begin{equation}
	\label{eq:ProjectionSeries}
	G_0 f - \sum_{i_1} Q_{i_1} G_0 f 
	+ \sum_{i_1} \sum_{ i_2 \neq i_1} Q_{i_1} Q_{i_2} G_0 f
	- \sum_{i_1} \sum_{ i_2 \neq i_1} \sum_{ i_3 \neq i_2} Q_{i_1}  Q_{i_2} Q_{i_3}G_0 f+ \dots.
\end{equation}

As mentioned in the introduction, we want to rewrite this series in terms of powers of a certain operator.
To do so, the key observation is that
\begin{equation}
	\label{eq:keyForPowers}
	\Phi_{k,i} = - Q_{i} \sum_{j \neq i} \Phi_{k-1,j} = -Q_{i} \Psi_{k-1}.
\end{equation}
This is due to the fact that 
\begin{equation}
	\label{eq:boundaryOfPsi}
	\Psi_{k-1} = \Phi_{k,i} \in B_i,
\end{equation}
which follows inductively from the definition of $\Psi_{k}$ and $\Phi_{k,i}$. 

Therefore, we have 
\[
	\Psi_{k+1} = -\sum_i Q_{i} \Psi_k,
\]
and thus, the partial sums of the scattering series are given by
\begin{equation}
	\label{eq:SeriesAsPowers}
	\left (1- \sum_i Q_{i}\right)^n G_0 f.
\end{equation}

\begin{definition}
	\label{def:L}
	The operator $ L \colon H^1(\IR^3) \supset \mathcal{D}(L) \to H^1(\IR^3) $ is defined as
	\begin{equation}
		\label{eq:definitionL}
		L = \sum_i Q_i.
	\end{equation}
	The domain $\mathcal{D}(L) $ of this operator consists of all function $u \in H^1(\IR^3)$ such that 
	the series $\sum_i Q_i$ exists.
\end{definition}

\begin{remark}
	We will show below (cf. Proposition \ref{pro:Aone}) that $L$ is a bounded operator in the whole of $H^1(\IR^3)$.
	As mentioned in the introduction, this is due to the exponential decay in the fundamental solution of the
	screened Poisson equation and fails for the Poisson equation.
\end{remark}

\begin{remark}
	\label{rem:LselfAdjoint}.
	We note that $\mathcal{D}(L) = H^1(\IR^3)$ implies that $ L $ is a nonnegative self-adjoint operator,
	since the operators $ Q_i $ are orthogonal projections
\end{remark}

\begin{theorem}
	\label{th:SeriesConvergentImpliesSolution}
	\begin{enumerate}[(i)]
	 \item If the series \eqref{eq:ProjectionSeries} obtained by the Method of Reflections is absolutely convergent, then it yields a solution to the Dirichlet problem \eqref{S1E8}. 
	 
 \item	The series \eqref{eq:ProjectionSeries} is absolutely convergent for every $f \in H^{-1}(\IR^3)$ if the operator $L$ from 	Definition \ref{def:L} is a bounded operator on $H^1(\IR^3)$ with $\|L\| < 2$.
 The series \eqref{eq:ProjectionSeries} is convergent for every $f \in H^{-1}(\IR^3)$, then
 $L$ defines a bounded operator on $H^1(\IR^3)$ with $\|L\| \leq 2$.
 
 \item Assume $L$ is a bounded operator on $H^1(\IR^3)$ with $\|L\| < 2$, and $L$ has a spectral gap, i.e., 
 \[
 	\inf \{\lambda \in \sigma(L) \backslash \{0\}\} = c > 0,
 \]
 where $\sigma(L)$ denotes the spectrum of $L$. Then, there exists $\eps < 1$ depending only on $\|L\|$ and $c$ such that
 \begin{equation}
 	\|(1-L)^n G_0 f - u \|_{H^1_\xi(\IR^3)} \leq \eps^n \|f\|_{H^{-1}_\xi(\IR^3)} \qquad \text{for all} \quad f \in H^{-1}(\IR^3),
 	\label{eq:expConvergence}
 \end{equation}
 where $u$ denotes the solution to the Dirichlet problem \eqref{S1E8}.
	\end{enumerate}
\end{theorem}

\begin{proof}
	As above, we denote the partial sums of the series \eqref{eq:ProjectionSeries} by $\Psi_n$.
	Since $(-\Delta +\xi^{-2}) Q_i v = 0$ for all $v \in H^1(\IR^3)$ (cf. \eqref{eq:characterizationOfQ}), it follows
	\[
		(-\Delta +\xi^{-2}) \Psi_n = f \in \IR^3 \backslash K.
	\]
	Thus, this equation is also satisfied by the limit.
	By \eqref{eq:boundaryOfPsi} we have $\Psi_n = \Phi_{n+1,i} \to 0$ in $B_i$ since $\Phi_{n+1,i}$
	appears in the series \eqref{eq:ProjectionSeries} which we assumed to be absolutely convergent.
	This implies that the limit indeed solves \eqref{S1E8}.
	
	To prove the second statement, we observe that by \eqref{eq:SeriesAsPowers}, 
	the partial sums of the series \eqref{eq:ProjectionSeries}
can be written as $(1-L)^n G_0 f$. Since $G_0$ is an isometry, 
these partial sums only exist if $\mathcal{D}(L) = H^1(\IR^3)$. Then, by Remark \ref{rem:LselfAdjoint},
$L$ is a nonnegative self-adjoint operator. 
Thus, by the spectral theorem (for unbounded self-adjoint operators), up to an isometry, 
	$ L $ is a multiplication operator $T$ on $ H := L^2_\nu(X)$ 
	for some measure space $(X,\mathcal{A},\nu)$, i.e., there exists a function $f \in L^\infty_\nu (X)$ such that
	$ T \varphi = f \varphi $ for all $ \varphi \in L^2_\nu(X)$.
	Thus, $(1-L)^n G_0 f$ corresponds to 
	\[
		(1-f)^n \varphi
	\]
	which converges iff 
	\[
		-1 < (1-f) \leq 1 ~ \nu\text{-a.e.}
	\]
	Since $L$ is nonnegative, this is equivalent to $f < 2 $, $\nu$-a.e., and hence,
	a sufficient condition for convergence is $\|L \| < 2$, and a necessary condition is $\|L \| \leq 2$.
	
	If, in addition, $L$ has a spectral gap, then for $\nu$-a.e. $x$, $f(x) = 0$ or $ f(x) \geq c$ and \eqref{eq:expConvergence} follows.
\end{proof}

\begin{remark}
	\label{rem:kerL}
	It is essential to observe the following. If $L$ defines a bounded operator on $H^1(\IR^3)$ with $\|L\| < 2$,
	then $(1-L)^n$ converges to the orthogonal projection to the kernel of $L$. 
	Indeed, by decomposing any $u \in H^1(\IR^3)$ into $u = u_1 + u_2$, where $u_1 \in \ker L$ and
	$u_2 \in (\ker L)^\perp$, we see that $(1-L)^n u_2 = u_2 $ and $ (1-L)^n u_1 \to 0$ using the spectral theorem
	as in the proof above.

	We recall that $L = \sum_i Q_i$, where $Q_i$ are orthogonal projections to
	$ H_0^1(\IR^3 \backslash \overline{B_i})^\perp$. Therefore,
\begin{equation}
	\label{eq:kerLasIntersection}
	\ker L = \bigcap_i H_0^1(\IR^3 \backslash \overline{B_i}) = H_0^1(\IR^3 \backslash K) =: V.
\end{equation}
	Hence, the series \eqref{eq:ProjectionSeries} written as $(1-L)^n G_0 f$ converges to $P G_0 f$, where $P$
	denotes orthogonal projection to $V$. However, this is just a different way to see that the series indeed converges
	to the solution of Problem \eqref{S1E8}. Indeed, the fact that $P G_0 f$ solves Problem \eqref{S1E8} follows 
	directly from Lemma \ref{lem:scrpoibdry}.
\end{remark}

\subsection{Estimates for the Operator $L$}
\label{sec:AnEstimateForTheFirstOrderTermInTheScatteringSeries}

\begin{proposition}
	\label{pro:Aone}
	There exists a constant $C_1 > 0 $ depending only on $\alpha$ and $\kappa$ from Condition  
	\ref{cond:particlesNotToClose} and \ref{cond:radiusSmallerXi} such that 
	\begin{equation}
		\label{eq:Aone}
		 \| L \| \leq (1 + C_1 \xi^2\mu_0).
	\end{equation}
	
\end{proposition}

The key estimate for the proof of the above proposition is the following lemma. Roughly speaking, it states that correlations between $H^{-1}$ functions which are supported in the particles are controlled by the capacity density times the norms of the functions themselves.

\begin{lemma}
	\label{lem:locH^-1}
	Assume $(f_i)_{i \in I} \subset H^{-1}(\IR^3)$
	satisfies $ \supp f_i \subset \overline{B_i}$ for all $i \in I$.
	Then,
	\begin{equation}
		\label{eq:locH^-1}
		c \sum_i \| f_i \|_{H^{-1}_\xi(\IR^3)}^2 \leq \Big \| \sum_i f_i \Big \|_{H^{-1}_\xi(\IR^3)}^2 
		\leq (1+C_1 \mu_0\xi^2) \sum_i \| f_i \|_{H^{-1}_\xi(\IR^3)}^2,
	\end{equation}
	where $ c>0 $ is a universal constant and $C_1$ depends only on  $\alpha$ and $\kappa$ from Condition 
	\ref{cond:particlesNotToClose} and \ref{cond:radiusSmallerXi}.
\end{lemma}

For the proof we need the following lemma.

\begin{lemma}
	\label{lem:correlationest}
	Let $i,j \in I$. Assume $ f \in H^{-1}(\IR^3) $ is supported in $ \overline{B_j} $. 
	Then, there exists a function $ v \in H^1_0(B_{\kappa r_i}(x_i)) $ such that $ v = G_0 f $ in $ B_i$, and 
	\begin{equation}
		\|v\|_{H^1_\xi(\IR^3)} \leq C \sqrt{r_i r_j} \frac{e^{-\frac{|x_i - x_j|}{\xi}}}{|x_i - x_j|} \|f\|_{H^{-1}_\xi(\IR^3)},
	\end{equation}
	for a constant $ C $ that depends only on  $\alpha$ and $\kappa$ from Condition 
	\ref{cond:particlesNotToClose} and \ref{cond:radiusSmallerXi}.
\end{lemma}

\begin{proof}	
	For $ z \in B_{\kappa r_i}(x_i)$, we define $ \theta \in C_c^\infty (B_{\kappa r_j}(z - x_j)) $ such that
	$ \theta = 1$ in $ B_{r_j}(z - x_j) $ and $ |\nabla \theta| \leq \frac{C}{r_j} $, 
	(where the constant depends on $\kappa$). 
	We use that $ f $ is supported in $ \overline{B_j} $. Therefore, using the fundamental solution \eqref{eq:FundamentalScreened},
	\begin{equation}
	\begin{aligned}
		|(G_0 f) (z)| &= |(W_\xi \ast f)(z)| = |((\theta W_\xi) \ast f)(z)| \\ 
		&= |\langle (\theta W_\xi)(z-\cdot),f \rangle| 
		\leq \|f\|_{H^{-1}_\xi(\IR^3)} \|\theta W_\xi \|_{H^1_\xi(\IR^3)},
	\end{aligned}
	\label{eq:ptw1}
	\end{equation}
	and
	\begin{equation}
		\label{eq:ptw2}
		|\nabla (G_0 f) (z)| \leq \| f \|_{H_\xi^{-1}(\IR^3)} \|  \theta \nabla W_\xi \|_{H^1_\xi(\IR^3)}.  
	\end{equation}
	Using Condition \ref{cond:particlesNotToClose} and \ref{cond:radiusSmallerXi}, we estimate
	\begin{align}
		\|\theta W_\xi \|_{H^1_\xi(\IR^3)} &\leq \| W_\xi \|_{H^1_\xi(B_{\kappa r_j}(z - x_j))} + 
		\frac{C}{r_j}\| W_\xi \|_{L^2(B_{\kappa r_j}(z - x_j))} \\
		&\leq C r_j^{3/2} e^{-\frac{|x_i - x_j|-\kappa r_j} {\xi}} \left( 
		\frac{1}{(|x - y|-\kappa r_j)^2} + \frac{1}{r_j(|x_i - x_j|-r_j)} + \frac{1}{\xi(|x_i - x_j|-r_j)}\right) \\
		& \leq C r_j^{1/2} \frac{ e^{-\frac{|x_i - x_j|}{\xi}}}{|x_i - x_j|},
	\end{align}
	and 
	\[
		\|\theta \nabla W_\xi \|_{H^1_\xi(\IR^3)} \leq  
		C r_j^{1/2}\frac{e^{-\frac{|x_i - x_j|}{\xi}}}{|x_i - x_j|^2}.
	\]
	Now, we use another cutoff function 
	$ \eta \in C_c^\infty (B_{\kappa r_i}(x_i)) $ such that
	$ \eta = 1$ in $ B_i $ and $ |\nabla \eta| \leq \frac{C}{r_i} $ to define $ v := \eta(G_0 f)$.
	Then, we get from the pointwise estimates on $ G_0 f $, \eqref{eq:ptw1} and \eqref{eq:ptw2}, 
	\begin{align}
		\| v\|_{H^1_\xi(\IR^3)} = \|\eta (G_0 f) \|_{H^1(\IR^3)} &\leq \| G_0 f \|_{H^1(B_{\kappa r_i}(x_i))} + 
		\frac{C}{r_i}\| G_0 f \|_{L^2(B_{\kappa r_i}( x_i))} \\
		&\leq C \sqrt{r_i r_j} \frac{e^{-\frac{|x_i - x_j|}{\xi}}}{|x_i - x_j|} \|f\|_{H^{-1}(\IR^3)}. \qedhere
	\end{align}
\end{proof}

\begin{proof}[Proof of Lemma \ref{lem:locH^-1}.]
	Let $ \eta_i \in C_c^\infty (B_{\kappa r_i}(x)) $ such that 
	$ \eta_i = 1 $ in 	$ B_i $ and $ |\nabla \eta_i | \leq \frac{C}{r_i} $. 
	Now, we observe that for all $ u \in H^1(\IR^3) $
	\[
		\| u \|_{L^2(B_{\kappa r_i}(x_i))} \leq \| u \|_{L^6(B_{\kappa r_i}(x_i))} \|1\|_{L^3(B_{\kappa r_i}(x_i))} 
		\leq  C r_i \| \nabla u \|_{L^2(\IR^3)},
	\]
	where we have used the Gagliardo-Nirenberg-Sobolev inequality $\|u\|_{L^6(\IR^3)} \leq C \| \nabla u \|_{L^2(\IR^3)}$.
	Hence,
	\begin{equation}
		\| \eta_i u \|_{H^1_\xi(\IR^3)} \leq \|u\|_{H^1_\xi(\IR^3)} + \frac{C}{r_i} \| u \|_{L^2(B_{\kappa r_i}(x_i))} 
		\leq C \| u \|_{H^1_\xi(\IR^3)}.
		\label{eq:cutoffest}
	\end{equation}
	On the other hand, denoting $f = \sum_i f_i$,
	\begin{align}
		\sum_i \| f_i \|_{H_\xi^{-1}(\IR^3)}^2 &= \sum_i\langle G_0 f_i,f_i \rangle = 
		\sum_i \langle  \eta_i G_0 f_i,f_i\rangle \\ &= \sum_i \langle \eta_i G_0 f_i,f \rangle 
		\leq \| f \|_{H^{-1}_\xi(\IR^3)} \Big \|\sum_i \eta_i G_0 f_i \Big \|_{H^1_\xi(\IR^3)}.
	\end{align}
By taking squares on both sides and using the fact that the balls $ B_{\kappa r_i}(x_i) $ are disjoint together with the preliminary estimate \eqref{eq:cutoffest}, we deduce
\[
	\bigg (\sum_i \| f_i \|_{H_\xi^{-1}(\IR^3)}^2 \bigg )^2 
	\leq C \| f \|_{H^{-1}_\xi(\IR^3)}^2 \sum_i  \|G_0 f_i \|_{H^1_\xi(\IR^3)}^2.
\]
Since $ G_0 $ is an isometry, this yields the first inequality in \eqref{eq:locH^-1}.

	For the second inequality, we use again that $ G_0 $ is an isometry to get
	\begin{align}
		\Big \| \sum_i f_i \Big \|_{H^{-1}_\xi(\IR^3)}^2 &= \Big \| \sum_i G_0 f_i \Big \|_{H^1_\xi(\IR^3)}^2 \\ 
		& = \sum_i \| G_0 f_i \|_{H^1_\xi(\IR^3)}^2 
		+ \sum_i \sum_{j \neq i} (G_0 f_i, G_0 f_j)_{H^1_\xi(\IR^3)} \\
		&= \sum_i \| f_i \|_{H_\xi^{-1}(\IR^3)}^2 
		+ \sum_i \sum_{j \neq i} \langle  G_0 f_j , f_i\rangle.
	\end{align}
	Let $ i \neq j $. Since $ f_i $ is supported in $ \overline{B_i}, $ we have
	\[
		\langle  G_0 f_j, f_i \rangle = \langle  v, f_i \rangle,
	\]
	for any $ v \in H^1(\IR^3) $ such that $ v = G_0 f_j $ in $B_i$. Therefore, application of Lemma \ref{lem:correlationest}
	yields
	\[
		 |\langle  G_0 f_j , f_i\rangle| 
					\leq C \sqrt{r_i r_j} \frac{e^{-\frac{|x_i - x_j|}{\xi}}}{|x_i - x_j|} \|f\|_{H^{-1}_\xi(\IR^3)}.
	\]
	Finally, taking the sum in $ i $ and $ j $ and using 
	\[ 
		\sqrt{r_i r_j} \| f_i \|_{H^{-1}_\xi(\IR^3)} \| f_j \|_{H^{-1}_\xi(\IR^3)} 
		\leq \frac{1}{2} \left(r_i  \| f_i \|_{H_\xi^{-1}(\IR^3)}^2 + r_j\| f_j\|_{H_\xi^{-1}(\IR^3)}^2 \right)
	\]
	and symmetry in $ i $ and $ j $, we conclude using Condition \ref{cond:Capacity}
	\begin{align}
		\sum_i \sum_{j \neq i} \langle G_0 f_j , f_i\rangle
		& \leq C\sum_i \sum_{j \neq i} r_i \frac{e^{-\frac{|x_i - x_j|}{\xi}}}{|x_i - x_j|} \| f_i \|_{H^{-1}_\xi(\IR^3)}^2 \\
		& \leq C \sum_i \mu_0  d_i^3 \| f_i \|_{H^{-1}_\xi(\IR^3)}^2 \\
		& \leq C \sum_i  \| f_i \|_{H^{-1}_\xi(\IR^3)}^2 \int_{\IR^3} \frac{e^{-\frac{|z|}{\xi}}}{|z|} \dd z \\
		& \leq C \mu_0 \xi^2 \sum_i \| f_i \|_{H^{-1}_\xi(\IR^3)}^2. \qedhere
	\end{align}
\end{proof}

\begin{proof}[Proof of Proposition \ref{pro:Aone}.]
	 We choose an enumeration of the index set $I$ and define 
	\[
		L_N := \sum_{i=1}^N Q_i,
	\]
	where $Q_i$ was defined in \eqref{eq:Q}.
	From \eqref{eq:characterizationOfOrthogonal} we see that every function in the image of $G_0 ^{-1} Q_i$ 
	is supported in $B_i$.
	Using that $G_0$ is an isometry, Lemma \ref{lem:locH^-1} implies
	\begin{equation}
		\label{eq:estf}
		\begin{aligned}
		\|L^N u\|_{H^1_\xi(\IR^3)}^2 \leq (1+C_1 \xi^2 \mu_0) \sum_{i=0}^N \| Q_{i} u \|_{H^1_\xi(\IR^3)}^2 
		&= (1+C_1 \xi^2\mu_0) \sum_{i=0}^N (Q_{i} u,u)_{H^1_\xi(\IR^3)} \\
		&= (1+C_1 \xi^2\mu_0) (L^N u,u )_{H^1_\xi(\IR^3)}.
		\end{aligned}
	\end{equation}
	As a sum of orthogonal projections, $L^N$ is self-adjoint. Thus, by the spectral theorem
	for self-adjoint bounded operators, up to an isometry, 
	$ L^N $ is a multiplication operator $S$ on $ H := L^2_\nu(X)$ 
	for some measure space $(X,\mathcal{A},\nu)$, i.e., there exists a function $f \in L^\infty_\nu (X)$ such that
	$ S \varphi = f \varphi $ for all $ \varphi \in L^2_\nu(X)$.
	The estimate \eqref{eq:estf} above yields 
	\[
		\int_X f^2 \varphi^2 \dd \nu \leq (1+C_1 \xi^2\mu_0) \int_X f \varphi^2 \dd \nu,
	\]
	implying
	\[
		\|L^N_r\| = \|f\|_{L^\infty(X)} \leq 1 + C_1 \mu_0\xi^2.
	\]
	
	On the other hand, convergence of $L^N u$ holds for any $u \in H^1(\IR^3)$ that is compactly supported,
	because particles lying outside of the support of $u$ do not play any role.
	By an $\frac{\eps}{3}$-argument, $L u = \sum_{i = 1}^\infty Q_{i} u $ is well defined for all $u \in H^1(\IR^3)$
	and $\|L\| \leq 1 + C_1 \xi^2 \mu_0$.
\end{proof}

\begin{remark}
\label{rem:EstimateLSharp}
	The second estimate in \eqref{eq:locH^-1} is sharp in the following sense.
	For all particle configurations, $\|L\| \geq 1$, and there exist particle configurations satisfying Conditions
 \ref{cond:Capacity}, \ref{cond:particlesNotToClose}, and \ref{cond:radiusSmallerXi},  such that
\[
	\|L \| \geq c \xi^2 \mu_0,
\]
for a universal constant $c$.
\end{remark}

\begin{proof}
	Consider any particle configuration and a function supported in one particle, i.e., $u \in H^1_0(B_i)$
	for some $i \in I$. Then $u$ is a fixed point of the operator $L = \sum_i Q_i$, because
	$Q_i u = u$ and $Q_j u = 0$ for all $j \neq i$. Hence $\| L \| \geq 1$.
	
	The fact that also the capacity $\mu_0$ has to appear on the right hand side follows more or less directly from the
	definition of the electrostatic capacity. The capacity of a set $K$ is defined as
	\[
		\|\nabla v \|_{L^2(\IR^3 \backslash K)}^2,
	\]
	where $v$ is the solution to 	
	\begin{align}
		-\Delta v &= 0 \quad \text{in} ~ \IR^3 \backslash K, \\
		v &= 1 \quad \text{in} ~ K.
	\end{align}
	Now we consider particles distributed on a lattice with equal radius $r$, i.e., the set $K$ occupied by the particles
	is 
	\[
		K=\bigcup_{x\in\left(  d\mathbb{Z}\right)^{3}  }\overline{B_{r}\left(x\right)  }.
	\]
	We choose $d << 1$ and consider $u \in H^1(\IR^3)$ such that $ u = 1 $ in $B := B_1(0)$.
	Then, for each $x_i \in \left(d\mathbb{Z}\right)^{3} \cap B$, we have for $y \in \IR^3 \backslash B_i$
	\[
		(Q_i u)(y) = r e^{\frac{r}{\xi}} \frac{e^{-\frac{|y-x_i|}{\xi}}}{|y-x_i|},
	\]
	and thus, 
	\[
		\| Q_i u \|_{H^1(\IR^3)}^2 \geq \| \nabla Q_i u \|_{L^2(\IR^3)}^2 
		\geq r^2 \int_r^\infty \frac{e^{-\frac{s}{\xi}}}{s^2} \dd s \geq C \xi^2 r.
	\]
	Therefore, using again that $Q_i$ is an orthogonal projection,
	\[
		\|L_r\| \geq c (L_r u,u)_{H^1(\IR^3)} = c \sum_x (Q_x u,u)_{H^1(\IR^3)}
		\geq c  \!\! \sum_{x \in \left( d\mathbb{Z}\right)^{3}  \cap B}  \!\! \|Q_x u\|^2_{H^1(\IR^3)} 
		\geq c \xi^2  \!\! \sum_{x \in \left( d\mathbb{Z}\right)^{3}  \cap B} r,
	\]
	where we put the norm of $u$ into the constant because $u$ has been chosen independently of the particle distribution.
	Since the number of $x_i$ in $\left( d\mathbb{Z}\right)^{3}  \cap B$ is of order $ d^{-3} = \mu_0 r^{-1}$,
	we conclude $ \|L\| \geq c\mu_0$. 
\end{proof}

Using the bound on the norm of $L$ that we proved in Proposition \ref{pro:Aone} it follows from Theorem 
\ref{th:SeriesConvergentImpliesSolution} that the series \eqref{eq:ProjectionSeries} obtained by the
Method of Reflections converges to the solution of Problem \eqref{S1E8}.
Uniform convergence also follows from Theorem \ref{th:SeriesConvergentImpliesSolution} and the following Lemma.

\begin{lemma}
	\label{lem:LCoercive}
	There exists a constant $c_1 > 0$ depending only on $\kappa$ from Condition	\ref{cond:particlesNotToClose} such that
	\[
		 (L u,u)_{H^1(\IR^3)} \geq c_1 \| u\|_{H^1(\IR^3)}^2,
	\]
	for all  $u \in H_0^1(\IR^3 \backslash K)^\perp$.
\end{lemma}

\begin{proof}
	Let $ \eta_i \in C_c^\infty (B_{\kappa r_i(x_i)}) $ such that 
	$ \eta_i = 1 $ in 	$ B_i $  and $ |\nabla \eta_i | \leq \frac{C}{r_I} $. 
	Now, we observe that for all $ v \in H^1(\IR^3) $
	\[
		\| v \|_{L^2(B_{\kappa r_i}(x_i))} \leq \| v \|_{L^6(B_{\kappa r_i}(x_i))} \|1\|_{L^3(B_{\kappa r_i}(x_i))} 
		\leq  C r_i \| \nabla v \|_{L^2(\IR^3)},
	\]
	and hence,
	\begin{equation}
		\| \eta_x v \|_{H^1(\IR^3)} \leq \|v\|_{H^1(\IR^3)} + \frac{1}{r} \| v \|_{L^2(B_{2r}(x))} 
		\leq C \| v \|_{H^1(\IR^3)}.
	\end{equation}
	
	On the other hand, we know that every $u \in H_0^1(\IR^3 \backslash K)^\perp$ satisfies $-\Delta u + \xi^{-2} u = 0$ in $\IR^3 \backslash K$ (cf. Equation \eqref{eq:characterizationOfOrthogonal}).
	Thus, the variational form of this equation implies that $ u $ is the function of minimal norm 
	in the set $ X_u := \{ v \in H^1(\IR^3) \colon v = u ~ \text{in} ~ K\}$. 
	Clearly, $\sum_i \eta_i Q_x u \in X_u$, and hence,
	\begin{align}
		(L u , u)_{H^1(\IR^3)} &=  \sum_i ( Q_i u, u)_{H^1(\IR^3)} = \sum_i \|Q_i u \|_{H^1(\IR^3)}^2 \\
		&\geq c \sum_i \|\eta_i Q_i u \|_{H^1(\IR^3)}^2 
		= c \|\sum_i \eta_i Q_i u \|_{H^1(\IR^3)}^2 \geq c \|u\|^2_{H^1(\IR^3)}. \qedhere
	\end{align}
\end{proof}

\begin{proof}[Proof of Theorem \ref{SeriesPoissonScreened}]
By Proposition \ref{pro:Aone}, we have $\| L \| \leq 1 + C_1 \xi^2 \mu_0$.
Defining $C_0 := \frac{1}{2 C_1}$, we have $\|L\| \leq \frac{3}{2}$ if $\mu_0 \leq C_0 \xi^{-2}$.
 Furthermore, Lemma \ref{lem:LCoercive} implies 
\begin{equation}
	\label{eq:SpectralGap}
	\|L u\| \geq c_1 \|u\|
\end{equation}
for all $u \in H_0^1(\IR^3 \backslash K)^\perp$. By Remark \ref{rem:kerL}, we have
$\ker L = H_0^1(\IR^3 \backslash K)$. Thus, Estimate \eqref{eq:SpectralGap} implies that $L$ has a spectral gap.
Therefore, Theorem \ref{th:SeriesConvergentImpliesSolution} implies the exponential convergence 
 \begin{equation}
 	\|(1-L)^n G_0 f - u \|_{H^1_\xi(\IR^3)} \leq \eps^n \|f\|_{H^{-1}_\xi(\IR^3)} \qquad \text{for all} \quad f \in H^{-1}(\IR^3),
 \end{equation}
 for some $\eps$ depending only on $c_1$ and thus depending only  $\kappa$ from Condition 
	\ref{cond:particlesNotToClose}.
Since the norm $\| \cdot \|_{H^1_\xi(\IR^3)}$ is equivalent to the standard $H^1$-norm, this concludes the proof.
\end{proof}

\subsection{Convergence of a Modified Method of Reflections}

In the previous subsection, we proved that the series \eqref{eq:ProjectionSeries} obtained by the Method of Reflections 
converges for small capacities. 
Recall that the series is given by
\begin{equation}
	\label{eq:ReflectionSeriesCompressed}
	\lim_{n \to \infty} (1-L)^n G_0 f.
\end{equation}
First of all, we note that the series is indeed divergent if the capacity is sufficiently large.
Indeed, as shown in Remark \ref{rem:EstimateLSharp} the operator norm of $L$
diverges as the capacity tends to infinity and we have already observed in Theorem \ref{th:SeriesConvergentImpliesSolution} that the series is divergent if
the operator norm of $L$ is larger than $2$. 

Now we want to give the series a meaning for arbitrary capacities.
As seen in Remark \ref{rem:kerL}, the solution to Problem \eqref{S1E8}, which we want to obtain by the Method of Reflections, is given by $P G_0 f$, where
$P$ is the orthogonal projection to the kernel of $L$.
Therefore, the modification simply consists in replacing \eqref{eq:ReflectionSeriesCompressed} by
\begin{equation}
	\label{eq:ResummationCompressed}
	\lim_{n \to \infty} (1-\gamma L)^n G_0 f,
\end{equation}
with $\gamma := 1/ \|L\|$.
Using again the spectral theorem, we will show in Proposition \ref{pro:abstractProjection} below that
this ensures convergence to the solution to Problem \eqref{S1E8}. However, let us first give a heuristic explanation
why this can be expected.

We can give the following interpretation of the Method of Reflections using the representation \eqref{eq:ReflectionSeriesCompressed}.
To the solution of the equation without boundary conditions $G_0 f$, we add the sum of all
the correctors, which is $-L$. Doing this, we expect to push the function towards zero boundary conditions.
By iterating this, we hope to obtain a sequence converging to the solution to the Dirichlet problem \eqref{S1E8}.
However, if $G_0 f$ has the same sign in several particles that are close to each
other and sufficiently large (i.e., large capacity), then, the effect of $L$ is too large: The boundary conditions in each of those particles
are not only corrected by the corresponding projection operator, but they also undergo a push in the same direction by
the effect of all the other particles. In other words, we push in the right direction but too far.
Therefore, reducing the push by multiplying with $\gamma$ might solve this problem.

We can also give a purely geometrical interpretation.
Let $P$ denote the orthogonal projection to $\ker L $, and $Q$ the projection to its orthogonal complement.
We recall that $L$ is the sum of the operators $Q_i$ which are the orthogonal projections.
Let us denote their kernel by $V_i$. Then
\begin{equation}
	\ker L = \bigcap_i V_i =: V.
\end{equation}
If the subspaces $V_i$ were orthogonal to each other, than, we would have 
\[
	1-L = 1-\sum_i Q_i = 1 - Q = P,
\]  
and the convergence of $(1-L)^n$ to $P$ would trivially hold.

However, they are not orthogonal to each other. Indeed, the closer two particles are, the more they interact with each
other. Interaction of the particles, however, means lack of orthogonality. 
Therefore, the series diverges if there is too much interaction between particles close to each other -- 
too large capacity $\mu_0$ -- or if the interaction does not decay fast enough -- too large $\xi$.

In Figure \ref{fig:1}, we see what happens in the orthogonal complement $V^\perp$ if the angles between the subspaces
$V_i$ are small. We consider the simplest non-trivial case in which only two particles are present.
As we see in Figure \ref{fig:1}, $(1-L)x$ might end up on the other side of the origin then $x$.
In this example, $(1-L)x$ is still closer to the origin than $x$. This is a feature of the case of only two
subspaces since $\|L\| < 2$ as long as the subspaces $V_i$ have trivial intersection. Therefore, the Method of Reflections always yields a convergent sequence if there are only two particles and they do not intersect.
However, if more particles are present and the angles between the subspaces are sufficiently small, $(1-L)x$ will
be larger than $x$. In that case, adding a small parameter $\gamma$ in front of $L$ will solve this problem.
Indeed, as in Figure \ref{fig:1}, we can ensure that $(1-\gamma L)x$ lies on the same side of the origin as $x$ by 
choosing $\gamma < 1/\|L\|$.

\begin{figure}[ht]
\centering
\includegraphics[scale= 0.5]{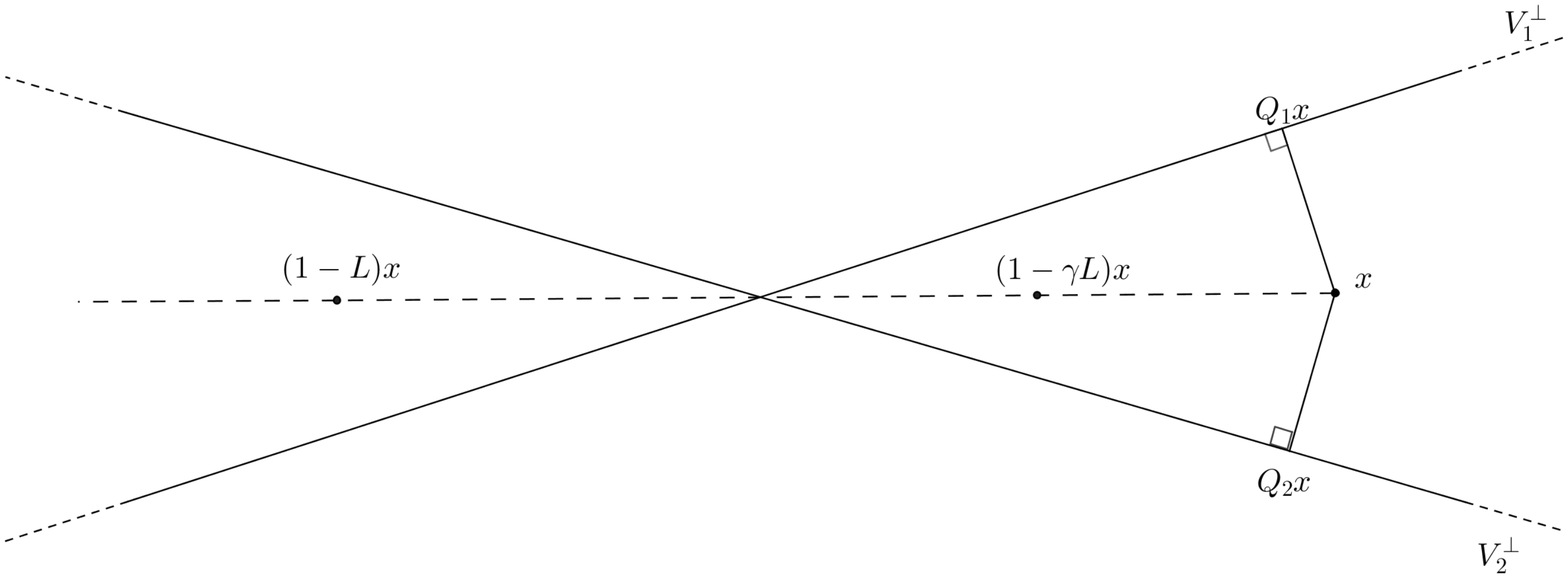}
\caption{}\label{fig:1}
\end{figure}

\begin{proposition}
	\label{pro:abstractProjection}
	Assume $ H $ is a Hilbert space and $ V_k \subset H$ are closed subspaces for $ k \in J$, where $j$ is a finite or
	countable index set. Define $Q_k$ to be the orthogonal projections from $ H$ to $V_k^\perp$.
	Let $ V = \cap_{k \in j} V_k$ and define $ P $ to be the orthogonal projection from $ H$ to $V$. 
	If $ S:= \sum_{k \in j} Q_k$ defines a bounded operator, then, for all $ 0< \gamma < \frac{2}{\|S\|} $,
	\[
		\lim_{M\to\infty} (1-\gamma S)^M = P,
	\]
	pointwise in $H$.
	Moreover, if $S$ is strictly positive in $V^\perp$, i.e.,
	there exists $ c > 0$ such that 
	\begin{equation}
		\label{eq:coercivityCondition}
		(Sx,x)_H \geq c \|x\|_H^2 \qquad \text{for all} \quad x \in V^\perp,
	\end{equation}
	then, 
	\begin{equation}
		\label{eq:exponentialConvergence}
		 \|(1-\gamma S)^M - P \| 
		\leq \max  \{1-\gamma c, \gamma \|S\| - 1  \}^M .
	\end{equation}
\end{proposition}

\begin{remark}
	To optimize the exponential convergence in \eqref{eq:exponentialConvergence}, one can choose
	\[
		\gamma = \frac{2}{\|S\| + c}.
	\]
\end{remark}

\begin{proof}
	By definition of $ S $, we have $ \ker S =V $. 
	Thus, $(1-\gamma S)^M x = x $ for all $x \in V$.  On the other hand, as $S$ is self-adjoint, we have
	$ \mathcal{R}(S) \subset (\ker S)^\perp = V^\perp $.
	
	We define $T$ as the restriction  of $ S $ to $ V^ \perp $ (in both the domain 
	and the range) satisfies $ \|1- \gamma T\| \leq \max\{1-\gamma c,\gamma \|S\| - 1\}$.
	Thus, it suffices to show that $(1-T)^n \to 0$ pointwise in $H$.
	
	Being a sum of orthogonal projections, $ S $ and also $T$ are self-adjoint operators.
	Hence, by the spectral theorem, we can assume that $ T $ is a multiplication operator on $ H = L^2_\nu(X)$ 
	for some measure space $(X,\mathcal{A},\nu)$, i.e., there exists a function $f \in L^\infty_\nu (X)$ such that
	$ T \varphi  =  f \varphi $ for all $ \varphi \in L^2_\nu(X)$.
	Since $T$ is positive and bounded by $\|S\|$, we have $0 < f \leq \|S\|$. Therefore,
	\begin{align}
		\|(1-\gamma T)^M \varphi\|^2_H 
		&= \int_X |\varphi|^2 (1-\gamma f)^{2M} \dd \nu  \to 0.
	\end{align}
	
	If in addition, \eqref{eq:coercivityCondition} holds, then $c < f \leq \|S\|$. Thus,
	\begin{align}
	 \|(1-\gamma T) \varphi\|^2_H 
		&= \int_X |\varphi|^2 (1-\gamma f)^{2} \dd \nu  \\
		&\leq \|1-\gamma f \|^2_{L^\infty_\nu(X)} \|\varphi\|^2_H \\
		&\leq \max\{1-\gamma c,\gamma \|S\|- 1 \} \|\varphi\|^2_H. \qedhere
	\end{align}
\end{proof}

\begin{corollary}
	\label{cor:solbyprojection}
	Let $C_1$ be the constant from Proposition \ref{pro:Aone}. 
	Then, for all particle configuration satisfying
	\begin{equation}
		C_1 \mu_0 \xi^{2} \leq C_2 ,%
	\end{equation}
	for some $C_2 < \infty$, there exists a constant $ \gamma_0 $, which depends only on $C_2$, with the following property.
	For all $\gamma \leq \gamma_0$,
	\[
		(1-\gamma L)^M \to P \qquad \text{in} ~ \mathcal{L}(H^1(\IR^3)) \quad \text{as} ~ M \to \infty,
	\]
	where $ P $ is the orthogonal projection from $ H^1(\IR^3) $ to $H_0^1(\IR^3 \backslash K)$. 
		
Moreover, there exists $\eps <1 $ depending only on $\kappa$, and $C_2$ such that
\[
	\|(1-\gamma_0 L)^M - P\|_{\mathcal{L}(H^1(\IR^3))} \leq C \eps^n,
\]
where $C$ depends only on $\xi$.
\end{corollary}

\begin{proof}
	We define $\gamma_0 = 1/(1+C_2)$. Proposition \ref{pro:Aone} implies $ \gamma_0 \leq 1/\|L\| $
	Then, the assertion follows directly from Proposition \ref{pro:abstractProjection} and Lemma \ref{lem:LCoercive}.
\end{proof}

\subsection{The Modified Method of Reflections as a Summation Method}
\label{sec:TheResummationOnTheLevelOfTheScatteringSeries}

\begin{lemma}
	\label{lem:combinatorics}
	Let $f \in H^{-1}(\IR^3)$. 
	Let $\Phi_n$ as in \eqref{eq:kthOrderCorrection} be the $n$-th order correction obtained by the 
	Method of Reflections. Then, for all $\gamma > 0$
	\begin{equation}
		(1-\gamma L)^M G_0 f = \sum_{n=0}^M q(n,M,\gamma) \Phi_n,
	\end{equation}
	where $q(0,M,\gamma) := 0 $, $q(n,M,\gamma) = 0 $ for $n > M$, and
	\[
			q(n,M,\gamma) = \frac{M!}{(M-n)!(n-1)!} \int_0^{\gamma} t^{n-1}(1-t)^{M-n} \dd t  
			= \frac{M!}{(M-n)!(n-1)!} B(\gamma;n,M-n+1),
	\]
	for $0 < n \leq M$. Here,  $ B $ denotes the incomplete Beta function. In particular, for all $\gamma >0$, and  $n\in\IN$ it holds
	\[
		\lim_{M\to\infty} q(n,M,\gamma) = 1.
	\]
\end{lemma}

\begin{proof}
	As we have seen in \eqref{eq:SeriesAsPowers}, it holds
	\[
		\sum_{n=0}^M  \Phi_n = (-L_r)^M G_0 f.
	\]
	By induction, this leads to the following identity
\begin{equation}
	\label{eq:combinatorics}
		(-L_r)^M G_0 f = \sum_{n=1}^M (-1)^{M-n} \binom{M-1}{n-1} \Phi_n.
\end{equation}

Expanding $(1- \gamma L)^M$ and using \eqref{eq:combinatorics} leads to $q(0,M,\gamma) = 1$, $q(n,M,\gamma) = 0 $ for $n > M$, and, for $0 < n \leq M$,
\begin{equation}
\begin{aligned}
	q(n,M,\gamma) &= \sum_{l=n}^M \binom{M}{l}\gamma^{l} (-1)^{l-n} \binom{l-1}{n-1} \\
	&=  (-1)^n \sum_{l=n}^M \frac{M! }{l(M-l)!}\frac{(-\gamma)^{l}}{(n-1)!(l-n)!} \\
	&= (-1)^n \frac{M!}{(n-1)!} \sum_{k=0}^{M-n} \frac{1}{k+l}\frac{ (-\gamma)^{k+l}}{(M-n-k)!k!}.
\end{aligned}
\end{equation}
Defining 
\[
	\psi(z) := \sum_{k=0}^{M-n} \frac{1}{k+n}\frac{z^{k+l}}{(M-n-k)!k!},
\] 
we find
\begin{align}
	\frac{d}{dz} \psi(z) &= \sum_{k=0}^{M-n} \frac{z^{k+n-1}}{(M-n-k)!k!} \\
	&= \frac{z^{n-1}}{(M-n)!}(1+z)^{M-n},
\end{align}
and hence,
\[
	\psi(z) = \frac{1}{(M-n)!} \int_0^z t^{n-1}(1+t)^{M-n} \dd t.
\]
Inserting this in the above equation, we finally get
\begin{equation}
\begin{aligned}
	\sum_{l=n}^M \binom{M}{l}\gamma^{l} (-1)^{l-n} \binom{l-1}{n-1} 
	&= (-1)^n \frac{M!}{(M-n)!(n-1)!} \int_0^{-\gamma} t^{n-1}(1+t)^{M-n} \dd t \\
	&= \frac{M!}{(M-n)!(n-1)!} \int_0^{\gamma} t^{n-1}(1-t)^{M-n} \dd t \\
	&= \frac{M!}{(M-n)!(n-1)!} B(\gamma;n,M-n+1).
\end{aligned}
\end{equation}

\end{proof}

\begin{proof}[Proof of Theorem \ref{CapOrderOne}]
The result is a direct consequence of Corollary \ref{cor:solbyprojection} and Lemma \ref{lem:combinatorics}.
\end{proof}

%% file: PoissonReflection.tex
\section{The Poisson Equation}
\label{sec:Poisson}

In order to directly apply the method to the Poisson equation, 
we need to change the spaces that we work in to make it possible to solve the Poisson equation in the whole space.

\begin{definition}
	We define the homogeneous Sobolev space $\dot{H}^1(\IR^3)$ as the closure of $ C_c^\infty(\IR^3) $ with respect to the $L^2$-norm of the gradient
	and denote its dual by $\dot{H}^{-1}(\IR^3)$.
	Moreover, for an open set $\Omega \subset \IR^3$, we define the space $\dot{H}_0^1(\Omega)$ to be 
	$\{u \in \Hdot \colon u = 0 ~ \text{in}~ \IR^3 \backslash \Omega\}$.
\end{definition}

Note that, with these definitions, the Laplacian is an isometry from $\Hdot $ into $ \dot{H}^{-1}(\IR^3)$. 
Correspondingly to the previous section, we denote $G_0 = (-\Delta)^{-1}$.

The following lemma corresponds to Lemma \ref{lem:scrpoibdry}.

\begin{lemma}
	\label{lem:poibdry}
	Let $\Omega \subset \IR^3$ be open. 
	Then, for every $ f \in  \dot{H}^{-1}(\IR^3) $,	the problem
	\begin{equation}
	\begin{aligned}
		\label{eq:poissonBall}
		-\Delta u  &= f \quad \text{in} ~ \IR^3 \backslash \overline{\Omega}, \\
		u &= 0 \quad \text{in} ~ \overline{\Omega}
	\end{aligned}
	\end{equation}
	has a unique weak solution $ u \in \dot{H}^1(\IR^3) $.
	Moreover, the solution for Problem \eqref{eq:poissonBall} is given by
	\begin{equation}
		 P_{\Omega} G_0 f,
	\end{equation}
	where $ P_{\Omega} $ is the orthogonal projection from $ \dot{H}^1 (\IR^3) $ to the subspace
	$ \dot{H}^1_0(\IR^3 \backslash \overline{\Omega})$.
\end{lemma}

As before, we define 
\[
	Q_i = 1 - P_i,
\]
where $P_{i} := P_{B_i}$ are the projection operators from Lemma \ref{lem:poibdry}.
Moreover, we note as in \eqref{eq:characterizationOfOrthogonal} that $Q_i$ is the orthogonal projection to 
\[
	 \dot{H}^1_0(\IR^3 \backslash \overline{B_i})^\perp = 
	 \{ v \in \Hdot(\IR^3) \colon -\Delta v = 0 \text{ in } \IR^3 \backslash \overline{B_i}\}.
\]

As mentioned before, the operator $ \sum_i Q_i$, which we have denoted $L$ for the screened Poisson equation,
will in general not be a bounded operator for infinitely particles. This is due to the long range interactions
of the Laplacian.
Therefore, we use a spatial cutoff to define the operator $L$ for the Poisson equation.

\begin{definition}
	\label{def:pL}
	We define
	\[
		L := \sum_i e^{-|x_i|} Q_i.
	\]
\end{definition}

\begin{remark}
	The choice of the specific exponential cutoff was only made for definiteness and to make the proof of the estimate for $L$ (cf. Lemma \ref{lem:Lbounded}) as analogous to the screened Poisson equation as possible. 
	However, any cutoff that maps $\Hdot(\IR^3)$ to $ \dot{H}^{-1}(\IR^3)$ would work (note that $\Hdot(\IR^3)$ is not contained in
	$ \dot{H}^{-1}(\IR^3)$). In particular, we could choose a polynomial cutoff with sufficiently fast decay.
\end{remark}

\subsection{Convergence of the Modified Method of Reflections}
\label{sec:ConvergenceOfTheSolutionRepresentation}

\begin{lemma}
	\label{lem:Lbounded}
	The operator $L$ from Definition \ref{def:pL} is a well defined, bounded, nonnegative, self-adjoint operator
	on $\Hdot(\IR^3)$ with
	\[
		\|L\| \leq (1+C\mu_0),
	\]
	where the constant $C$ depends only on $\kappa$ from Condition 
	\ref{cond:particlesNotToClose}.
\end{lemma}

The proof follows the lines of the proof of the corresponding result for
the screened Poisson equation, Proposition
\ref{pro:Aone}. The only difference is that the exponential cutoff in the definition of $L$ replaces the
the exponential decay of the fundamental solution of the screened Poisson equation \eqref{eq:FundamentalScreened}. We omit the details of the proof. However, we state the lemma corresponding to Lemma \ref{lem:locH^-1} for further
reference.

\begin{lemma}
	\label{lem:pLocH^-1}
	Assume $ (f_i)_{i \in I} \subset \dot{H}^{-1}(\IR^3) $ 
	satisfy $ \supp f_i \subset \overline{B_i} $. Then,
	\[
		\Big \| \sum_i e^{-|x_i|} f_i \Big \|_{\dot{H}^{-1}(\IR^3)}^2 
		\leq (1+C\mu_0) \sum_i  e^{-|x_i|} \|f_i \|_{\dot{H}^{-1}(\IR^3)}^2,
	\]
	where the constant $C$ depends only on $\kappa$ from Condition
	\ref{cond:particlesNotToClose}.
\end{lemma}

As in Proposition \ref{pro:abstractProjection} we would like to prove convergence for 
\[
	(1-\gamma L)^n G_0 f = (1 - \sum_i \gamma e^{-|x_i|} Q_i)^n G_0 f.
\]
for sufficiently small $\gamma >0$.
The only difference is that, instead of putting the same small factor $\gamma$ in front of all the operators $Q_i$,
we now have factors depending on the particle position due to the spatial cutoff $e^{-|x_i|}$ in Definition \ref{def:pL}.
Thus, we will see in Proposition \ref{pro:nonuniformAbstractProjection} below, that convergence to the desired solution
still holds for sufficiently small $\gamma$.
However, due to the spatial cutoff, $L$ lacks the coercivity on $\dot{H}_0^1(\IR^3 \backslash K)^\perp$ the analogous of which 
we had in the case of the screened Poisson equation (cf. Lemma \ref{lem:LCoercive}): Clearly, if 
$u \in \dot{H}_0^1(\IR^3 \backslash K)^\perp$ is only non-zero in particles very far away from 
the origin, then, $\|Lu\|_{\Hdot}$ is very small compared to $\|u\|_{\Hdot}$. 
Hence, we cannot expect any result about uniform convergence of $(1-\gamma L)^n G_0$ from a purely abstract 
argument as in Proposition \ref{pro:abstractProjection}. Indeed, the farther the mass of the source term
$f$ is away from the origin, the slower we expect the convergence to take place.

\begin{proposition}
	\label{pro:nonuniformAbstractProjection}
	Let $ H $ be a Hilbert space and $ V_k \subset H$ closed subspaces for $ k \in J$, where $J$ is a finite or
	countable index set. Define $Q_k$ to be the orthogonal projections from $ H$ to $V_k^\perp$.
	Let $ V = \cap_{k \in J} V_k$ and define $ P $ to be the orthogonal projection from $ H$ to $V$. 
	Assume $\gamma_k >0$, $k \in J$, are chosen such that $ S:= \sum_{k \in J} \gamma_k Q_k$ defines a bounded operator with
	$\|S\| < 2$.
	Then,
	\[
		\lim_{M\to\infty} (1-S)^M   = P,
	\]
	pointwise in $H$.
	
	If $\|S\| \leq 1$, then for all $x \in H$,
	\begin{equation}
		\label{eq:posLargerNorm}
		( S x, x)_H \geq \| S x \|_H^2,
	\end{equation}
	and 
	\begin{equation}
		\label{eq:monotonPos}
		(S(1-S)x,(1-S)x)_H \leq (Sx,x)_H.
	\end{equation}
\end{proposition}

\begin{proof}
	The statement about convergence is proven in the same way as in Proposition \ref{pro:abstractProjection}.

	Observe that estimates \eqref{eq:posLargerNorm} and \eqref{eq:monotonPos} are trivially satisfied in $V$.
	We define again $T$ as the restriction  of $ S $ to $ V^ \perp $ (in both the domain 
	and the range). Using the spectral theorem, we can assume $T$ to be a multiplication operator on
	$ H = L^2_\nu(X)$ 
	for some measure space $(X,\mathcal{A},\nu)$, i.e., there exists a function $f \in L^\infty_\nu (X)$ such that
	$ T \varphi $ = $ f \varphi $ for all $ \varphi \in L^2_\nu(X)$.
	By assumption, we know $0 < f \leq 1$. Therefore,
	\[
		(T \varphi, \varphi)_H = \int_X f \varphi^2 \dd \nu \geq  \int_X f^2 \varphi^2 \dd \nu = \| T \varphi \|_H^2,
	\]
	and
	\[
		(T (1-T) \varphi, (1-T)\varphi)_H = \int_X f(1-f)^2 \varphi^2 \dd \nu  \leq \int_X f \varphi^2 \dd \nu = 
		(T \varphi, \varphi)_H. \qedhere
	\]
\end{proof}

\begin{proof}[Proof of Theorem \ref{ConvWholeSpace}]
	We define $ \gamma_0 \leq 1/\|L_r\| $. Proposition \ref{pro:Aone} ensures that this is possible 
	in such a way that $\gamma_0$ depends only on $\mu_0$ and $\kappa$.   
	Then, the assertion follows directly from Proposition \ref{pro:nonuniformAbstractProjection} and
	Lemma \ref{lem:poibdry}.
\end{proof}

\subsection{The Modified Method of Reflections on the Level of the Original Series}

In this subsection, we will show how to compute the expansion of the term $(1-\gamma L)^n$ in order to obtain a
series similar to the original series obtained by the Method of Reflections \eqref{eq:ProjectionSeries}.
This is not only interesting in itself, but will be used to derive the homogenization results
Theorem \ref{HomogLambZero} and \ref{HomogStokes} in Section 4.

This leads to the following definition and lemma.

\begin{definition}
\label{def:A_beta}
	Let $n\in \IN_\ast$ and $ \beta \in \IN_\ast^n$, where we denote 
	$\IN_\ast := \IN \backslash \{0\}$. Then, we define the operator $A_\beta \colon \Hdot(\IR^3) \to \Hdot(\IR^3)$ by
	\begin{align}
		A_\beta = \sum_{i_1} e^{-\beta_1|x_{i_1}|} Q_{i_1} \sum_{i_2 \neq i_1} e^{-\beta_2|x_{i_2}|} Q_{i_2} \cdots
									\!\!	\sum_{i_n \neq i_{n-1}} \!\! e^{-\beta_n|x_{i_n}|} Q_{i_n}.
	\end{align}
\end{definition}

\begin{lemma}
	\label{lem:powersOfL}
	For all $n \in \IN_\ast $, the following identity holds
	\[
		(L_r)^n = \sum_{l=1}^n \sum_{\substack{\beta \in \IN_\ast^l \\ |\beta| = n}}A_\beta^{(r)}.
	\]
	In particular, for all $\beta \in \IN_\ast^n$, $A_\beta$ is a bounded operator with
	\[
		\|A_\beta\| \leq (1 + C \mu_0)^n,
	\]
	where $C$ is a universal constant.
\end{lemma}

\begin{proof}
	For $n=1$, the assertion is trivial. 
	Let $n \geq 2$ and $\beta \in \IN_\ast^n$. We write $\beta = (\beta_1,\beta')$ for some $\beta' \in \IN_\ast^{n-1}$.
	Using $Q_x^2 = Q_x$, it is easy to see that 
	\[
		L_r A_\beta = A_{(1,\beta)} + A_{(\beta_1 + 1 ,\beta')}.
	\]
	Observe that for every $1 \leq l \leq n+1$ and every $\gamma \in \IN_\ast^l$ with $|\gamma| = n+1$, either 
	$\gamma_1 = 1$, then, there exists a unique $\beta \in \IN_\ast^{l-1}$ with $|\beta| = n$ such that
	$\gamma = (1,\beta)$, or $\gamma_1 > 1$, then, $l \leq n$, and there exists a unique $\beta \in \IN_\ast^l$ with 
	$|\beta| = n$ such that $\gamma = (\beta_1 + 1,\beta')$.
	Therefore, the assertion for $n$ follows from the one for $n-1$.
	
	For $\beta \in \IN_\ast^n$ with $\beta_j =1$ for all $1 \leq j \leq n$, 
	the estimate for the operators $A_\beta$ follows directly from the bound on $L$ (see Lemma \ref{lem:Lbounded})
	and the identity that we just have proven, since all the operators $Q_{i}$ are positive.
	For general $\gamma \in \IN_\ast^n$, we clearly have $\|A_\gamma\| \leq \| A_\beta\|$ if $\beta $ is chosen
	as above. This concludes the proof.
\end{proof}

%% file: Homogenization.tex
\section{Homogenization}
\label{sec:Homogenization}

In this section, we will consider particles of equal radii $r$ with centers on the lattice $ \Gamma := (d \IZ)^3$.
Then, Condition \ref{cond:Capacity} is satisfied with $\mu_0 = r d^{-3}$. 
For the homogenization, it is convenient to include the factor $4 \pi$ in the capacity density, which we define as
\[
	\mu := 4 \pi r d^{-3}
\]

We are interested in the limiting behavior of Problem \eqref{S1E1} for $r,d \to 0$ and fixed $\mu$.
Thus, throughout this section, we will consider $\mu$ as a fixed quantity.
Since for fixed $\mu$ Condition \ref{cond:particlesNotToClose} will be satisfied if $r$ is sufficiently small,
we will always assume that $r$ is chosen in such a way.
In the following, we will use $\Gamma$ as the index set for the particles (i.e., we index them by their space position). Moreover, since for fixed $\mu$, the particle configuration does only depend on $r$, 
we will write an index $r$ to indicate this dependence,
e.g., we write
\[
	K_r = \bigcup_{x \in \Gamma} \overline{B_x}.
\]

\subsection{A Poincaré Inequality for Perforated Domains}

An important feature of this regular particle distribution is that Problem \eqref{S1E1}
admits a unique solution in $H^1(\IR^3)$ for sources $f \in H^{-1}(\IR^3)$, instead of solutions only 
in $\Hdot(\IR^3)$ for sources in  $\dot{H}^{-1}(\IR^3)$.
This is due to the existence of a Poincaré inequality in the space $H^1_0(\IR^3 \backslash K)$.

We first notice the following local Poincaré inequality.

\begin{lemma}
	\label{lem:poincare}
	Assume $ z \in \IR^3$, $ R> \rho > 0 $ and $ u \in H^1(B_R(z))$ such that $ u = 0 $ in 
	$B_\rho(z)$. Then, the following Poincaré inequality holds:
	\[
		\|u\|_{L^2(B_R(z))}^2 	\leq \frac{R^3}{\rho} \|\nabla u\|^2_{L^2(B_R(z))}.
	\]
\end{lemma}
\begin{proof}
	It suffices to prove the estimate for $ z= 0 $ and for smooth functions.
	Let $ \varphi \in C^1(B_R(0)) $ such that $ \varphi \equiv 0 $ in $B_\rho(0)$.
	Then, denoting the unit sphere in $ \IR^3$ by $ S^2$ we have for every $ x \in S^2 $ and every $ t \in (\rho,R)$
	\[
		|\varphi(tx)| \leq \int_\rho^R |\nabla \varphi(sx)| \dd s.
	\]
	Thus,
	\begin{align}
		\int_{B_R(0)} |\varphi|^2 \dd y
		& \leq \int_{S^2} \int_\rho^R t^2 \left( \int_\rho^R |\nabla \varphi(sx)| \dd s \right)^2 \dd t \dd x \\
		& \leq \frac{1}{3} (R^3 - \rho^3) \int_\rho^R \frac{1}{s^2} \dd s 
		\int_{S^2} \int_\rho^R s^2 |\nabla \varphi(sx)|^2 \dd s \dd x\\
		& \leq \frac{R^3}{\rho} \int_{B_R(z)} |\nabla \varphi|^2 \dd y. \qedhere
	\end{align}
\end{proof}

\begin{corollary}
	\label{cor:PoincarePerforated}
	All $ u \in H^1_0(\IR^3 \backslash K_r)$ satisfy
	\[
		\|u\|^2_{L^2(\IR^3)} \leq  C \mu^{-1} \| \nabla u \|^2_{L^2(\IR^3)}
	\]
	for a universal constant $C$.
\end{corollary}

\begin{corollary}
\label{cor:Existence}
	For all $f\in H^{-1}(\IR^3)$, there exists a unique weak solution $u \in H^1(\IR^3)$
	to the problem
	\begin{equation}
	\label{eq:poissonPerforated}
		\begin{aligned}
		-\Delta u &= f \quad \text{in} ~ \IR^3 \backslash K_r, \\
		u &= 0 \quad \text{in} ~ K_r,
		\end{aligned}
	\end{equation}
	which satisfies 
	\[
		\| u \|^2_{H^1(\IR^3)} \leq (1 + C \mu^{-1}) \| f\|^2_{H^{-1}(\IR^3)}.
	\]
\end{corollary}

\subsection{The Main Idea of the proof}

In order to explain the idea how we are going to prove the homogenization result, we need the following definition.

\begin{definition}
\label{def:monopole}

For a particle with radius $r$ at position $ x \in \Gamma_r $, we define 
the operator $T_{x}$ from  $\Hdot$ to $\dot{H}^{-1}(\IR^3)$ by means of
	\[
		Q_x = G_0 T_x.
	\]
Moreover, we define $	M_{x} \colon \dot{H}^{1}(\IR^3) \to \dot{H}^{-1}(\IR^3)$ to be
the uniform charge density approximation of $T_{x}$,
	\begin{align}
		(M_{x} u)(y) =  \frac{(u)_{x,r}}{r} \mathcal{H}^2 |_{\partial B_r(x)}.
	\end{align}
	Furthermore, we define $\tilde{Q}_x = G_0 M_{x,r}$ to be the induced approximation for $Q_x$.
	
	The uniform charge density approximations of the operators $ A_\beta^{(r)} $ from Definition \ref{def:A_beta}
	are defined by
	\begin{equation}
		\begin{aligned}
			M_\beta^{(r)} &:= \sum_{x_1} e^{-\beta_1|x_1|} \tilde{Q}_{x_1} 
			\sum_{x_2 \neq x_1} e^{-\beta_2|x_2|} \tilde{Q}_{x_2} \cdots 
			\!\! \sum_{x_n \neq x_{n-1}} \!\! e^{-\beta_n|x_n|} \tilde{Q}_{x_3}.
		\end{aligned}
	\end{equation}
\end{definition}

\begin{remark}
	Note that both $T_x$ and $M_x$ implicitly depend on $r$. 
\end{remark}

\begin{remark}
	\label{lem:charT}
	For $u \in H^1(\IR^3)$,  $ T_{x} u $ is supported in 
	$ \overline{B_x} $.
	Since $T_{x} =  G_0^{-1} Q_{x}$, and $Q_{x} $ is the orthogonal projection to  $H_0^1(\IR^3 \backslash \overline{B_x})^\perp$, this follows directly from the characterization \eqref{eq:characterizationOfOrthogonal}.
\end{remark}

To understand the meaning of the operator $T_{x}$, we take any potential $u \in \Hdot(\IR^3)$ and
denote $f:=G_0^{-1} u$ the source corresponding to $u$.
Moreover, we denote  $g= T_{x} u$.
Then,  adding $g$ to $f$, 
gives a source $f+g$, which corresponds to a potential $v := G_0(f+g)$ that solves
\begin{align}
		-\Delta v &= f \quad \text{in} ~ \IR^3 \backslash \overline{B_x}, \\
		v &= 0 \quad \text{in} ~ \overline{B_x}.
\end{align}

We can also draw the following analogy to electrostatics.
In this context,
$T_{x} G_0 f$ gives the charge density  that is induced by $f$ in $B_x$ if 
$B_x$ represents a grounded conductor (surrounded by vacuum).

With this definition the original series obtained by the Method of Reflection \eqref{eq:ProjectionSeries} becomes, 
\begin{equation}
\label{eq:ScatteringSeries}
	G_0 - \sum_{x_1} G_0 T_{x_1} G_0 
	+ \sum_{x_1} \sum_{ x_2 \neq x_1} G_0 T_{x_1} G_0 T_{x_2} G_0
	- \dots,
\end{equation}
This is how the series appears in \cite{Kirp}, where $T_{x}$ is called a scattering operator.
In this paper, the Method of Reflection is interpreted as 
a scattering process. 
Viewing $G_0$ as some kind of propagator, \eqref{eq:ScatteringSeries} inherits the interpretation
of the potential due to a source which propagates according to $G_0$ and scattered at the particles by $T_{x}$.

We want to give an heuristic explanation for the homogenization result Theorem \ref{HomogLambZero}.
To do so, let us pretend for the moment that the series \eqref{eq:ScatteringSeries} exists,
and that all the operators are well defined on $H^1(\IR^3)$ (instead of $\Hdot(\IR^3)$).
Moreover, let us assume that we already know that in the limit $r \to 0$,
we can replace the operator $T_x$ by $M_x$ in Definition \ref{def:monopole}. 
Using the definition of $M_x$ and recalling the fixed value of the capacity density $\mu = 4 \pi r d^{-3}$, 
the series $ \sum M_x u$ can be interpreted as a Riemann sum for $\mu u$, leading to
\[
	\sum_{x} T_x u \approx \sum_{x} M_x u \rightharpoonup \mu J u \quad \text{in} ~ H^{-1}(\IR^3),
\]
as $r \to 0$, where $J$ is the inclusion from $H^1(\IR^3)$ to $H^{-1}(\IR^3)$.

Therefore, the first order term in the series  \eqref{eq:ScatteringSeries} converges to $ (- G_0 J) G_0 f$. 
It seems plausible that the higher order terms converge weakly to 
$(-\mu G_0 J)^k G_0 f$.
Thus, the weak limit of the sequence of solutions is formally given by
\[
	\sum_{k=0}^\infty (-\mu G_0 J)^k G_0 = (1+ \mu G_0 J)^{-1} G_0 = (-\Delta + \mu J)^{-1},
\]
which is the desired result. 

Since the series \eqref{eq:ScatteringSeries} is in reality divergent,
we use the modified version
\begin{equation}
\label{eq:seriesForHomogenization}
	(1 - \gamma L_r)^n G_0 f,
\end{equation}
which we already know to converge to the solution of \eqref{S1E1}.
We want to expand \eqref{eq:seriesForHomogenization} in powers of $L$ and then to take the weak limit in each of the resulting terms separately.
However, one has to take into account
that the weak limit is not interchangeable with taking powers. Therefore it turns out,
that it is convenient to use Lemma \ref{lem:powersOfL} in order to write
$(L_T)^n$ as a sum of terms such that no particle appears back to back with itself.

Somewhat surprisingly, the exponential cutoff in the definition of the operator $L$ does not cause much trouble
when computing the weak limit. The only difference to the heuristic reasoning above is
 that some additional combinatorial identities are needed.

\subsection{Weak Limits of Powers of $L$}

Since the inclusion map from $\dot{H}^1(\IR^3)$ to $\dot{H}^{-1}(\IR^3)$ is not well defined, we need the following replacement.

\begin{definition}
\label{def:X}
	We define $X$ to be the following subspace of $\Hdot$.
	\begin{align}
		X &:= \{ u \in \dot{H}^1(\IR^3) \colon u = -\Delta v \text{ for some } v \in \dot{H}^1(\IR^3) \}. 
	\end{align}
	Moreover, we define $J \colon X \to \dot{H}^{-1}(\IR^3)$	by means of
	\[
		 \langle Ju,w \rangle = (\nabla v, \nabla w)_{L^2(\IR^3)} \qquad \text{for all} \quad w \in \Hdot,
	\]
	where $ v \in \dot{H}^1(\IR^3) $ is the solution to $ -\Delta v = u $.
\end{definition}

\begin{remark}
	Note that $J$ can be viewed as the inclusion map, since $\langle Ju,w \rangle = \int_{\IR^3} u w \dd x$,
	whenever the latter is well defined.
\end{remark}

\begin{lemma}
	\label{lem:operatorA}
	The operator $A \colon \dot{H}^1(\IR^3) \to \dot{H}^1(\IR^3) $,
	\begin{align}
		(Au)(x) = e^{-|x|} u(x),
	\end{align}
	is a bounded linear operator with range $\mathcal{R}(A) \subset X$.
	Moreover, the composition $JA$, where $J$ is the inclusion operator from Definition \ref{def:X}, is
	a bounded operator from $\dot{H}^1(\IR^3)$ to $\dot{H}^{-1}(\IR^3)$.
\end{lemma}

\begin{proof}
We observe that the range of $A$ satisfies 
$\mathcal{R}(A) \subset \dot{H}^1(\IR^3) \cap L^{6/5}(\IR^3) \subset X$. The first inclusion follows from 
the Gagliardo-Nirenberg-Sobolev inequality $ \| w \|_{L^6(\IR^3)} \leq C \|\nabla w\|_{L^2(\IR^3)}$
and Hölder's inequality.
The second one is deduced by the Gagliardo-Nirenberg-Sobolev inequality, too, 
since this implies boundedness of the functional $ F(w):= \int_{\IR^3} u w \dd x $ in $\Hdot$ if 
$ u \in \dot{H}^1(\IR^3) \cap L^{6/5}(\IR^3)$, providing in turn a solution
$ v \in \dot{H}^1(\IR^3) $ to $ - \Delta v = u$.

The second assertion follows from $\| Ju \|_{\dot{H}^{-1}(\IR^3)} = \| v \|_{\dot{H}^1(\IR^3)}$ and the reasoning above.
\end{proof}

\begin{proposition}
	\label{pro:pseqweaklimit}
	
	Let $ u \in \dot{H}^{1}(\IR^3) $ and $ n \in \IN_\ast $. Then, in the limit  $r \to 0$ with fixed $\mu$,
	\begin{equation}
		\label{eq:pseqweaklimit}
		L_r^n u \rightharpoonup \sum_{l=1}^n \sum_{ \substack{ \beta \in \IN_\ast^l \\ |\beta| = n}} 
								\Bigg(\prod_{j=1}^l  \mu G_0 J A^{\beta_j}\Bigg)  u  
							= \mu G_0 J A (  \mu G_0 J A + A)^{n-1} u =: R_n u	\quad \text{in} ~ \dot{H}^1(\IR^3).
	\end{equation}
	In particular, for all $ \gamma > 0 $ and all $M \in \IN$
	\[
		(1-\gamma L_{r})^M u \rightharpoonup \bigg(1+ \sum_{n=1}^M \binom{M}{n} (-\gamma)^n R_n \bigg) u =: S_M u 
		\quad \text{in} ~ \dot{H}^1(\IR^3)
	\]
\end{proposition}

The fact that the complicated looking weak limit of $L_r^n$ equals $R_n$ follows from the combinatorial consideration that,
expanding the power in the definition of $R_n$, each term in the sum on the right hand side will appear exactly once.

As mentioned above, the proof of Proposition \ref{pro:pseqweaklimit} is based on a Riemann sum argument using the
operators $T_x$ and $M_x$ from Definition \ref{def:monopole}.
This is not very difficult but technical. Therefore, we first show how to derive the homogenization result from
Proposition \ref{pro:pseqweaklimit} and the results from Section \ref{sec:Poisson}.

\begin{proposition}
	\label{pro:formalLimit}
	Let $M \in \IN$ and $S_M$ be the pointwise weak limit of $(1-\gamma L_{r})^M$ from Proposition \ref{pro:pseqweaklimit}.
	Then, for all $\mu>0$ there exists $\gamma_0 >0 $ such that, for all $\gamma \leq \gamma_0 $ and all 
	$f \in \dot{H}^{-1}(\IR^3)$,
	\[
		\lim_{M\to\infty} S_M G_0 f = u,
	\]
	where $u$ is the unique weak solution to
	\begin{equation}
		\label{eq:homoPDE}
		-\Delta u + \mu u = f \quad \text{in} ~\IR^3.
	\end{equation}
\end{proposition}

\begin{proof}

We observe that $\mu G_0 J + 1$ as an operator from $X$ to $\dot{H}^1(\IR^3)$ is invertible.
Indeed, we know that for any $f \in \dot{H}^{-1}(\IR^3) \subset H^{-1}(\IR^3)$,  Problem \eqref{eq:homoPDE}
has a unique weak solution  $ u \in H^1(\IR^3) \subset \dot{H}^1(\IR^3)$. Moreover, $ u = - \mu^{-1} \Delta (v-  u)$, 
where  $v \in \dot{H}^1(\IR^3)$ is the solution to $ -\Delta v = f $. 
Hence, we have $ u = (G_0^{-1} + \mu J)^{-1} f \in X$. Thus,
$(\mu G_0 J + 1)^{-1} = (G_0^{-1} + \mu J)^{-1} G_0^{-1}$.
Additionally, we see that $(\mu G_0 J + 1)^{-1}$ is a bounded operator since
for $u $ and $ f $ as above we have $ \| \nabla u \|_{L^2(\IR^3)} \leq \| f \|_{\dot{H}^{-1}(\IR^3)}$.

Therefore, inserting the definitions of $S_M$ and $R_n$ from the previous theorem, we deduce
	\begin{align}
	S_M = 1 + \sum_{n=1}^M \binom{M}{n}(-\gamma)^{n} R_n  
	& = 1 + \sum_{n=1}^M \binom{M}{n}(-\gamma)^{n} \mu G_0 J A ( \mu G_0 J A + A)^{n-1} \\[-3\jot]
	&= 1 + \mu G_0 J (\mu G_0 J + 1)^{-1} \sum_{n=1}^M \binom{M}{n}(-\gamma)^{n} ((\mu G_0 J +1)A)^n \\
	&= 1 + \mu G_0 J (\mu G_0 J + 1)^{-1} ((1-\gamma(\mu G_0 J +1)A)^M - 1). 
	\end{align} 

Next, we show that $(1-\gamma(\mu G_0 J +1)A)^M \to 0$ pointwise in $ \dot{H}^1(\IR^3)$ as $ M \to \infty$. 
First, by Lemma \ref{lem:operatorA}, we know that $G_0JA$ is a bounded operator. 
Second, $G_0JA$ is also a positive operator since
\[
	(G_0 J Au,u)_{\dot{H}^1(\IR^3)} = \langle JAu,u\rangle = \int Au \cdot u \dd x = \int e^{-|x|} |u(x)|^2 \dd x.
\]
Finally, $G_0JA$ is clearly self-adjoint since
\[
	(G_0 J Au,v)_{\dot{H}^1(\IR^3)} = \int Au \cdot v \dd x = \int Av \cdot u \dd x.
\]
Therefore, using the spectral theorem for bounded self-adjoint operators as in the proof of Proposition \ref{pro:abstractProjection},
we conclude $(1-\gamma(\mu G_0 J +1)A)^M \to 0$ pointwise in $ \Hdot$ for small enough $\gamma$.

Furthermore, 
\[
	\mu G_0 J (\mu G_0 J + 1)^{-1} = 1 - (\mu G_0 J + 1)^{-1},
\]
and hence, this is a bounded operator, as well.
Therefore, multiplying by $G_0$ from the right and taking the limit $M \to \infty$ yields
\begin{equation}
	(1 - (1-(\mu G_0 J+1)^{-1})) G_0 = (1 + \mu G_0 J)^{-1} G_0 
	= (G_0^{-1} + \mu J)^{-1} = ( -\Delta  + \mu)^{-1},
\end{equation}
which is the desired result.
\end{proof}

\subsection{Uniform Estimates and Proof of Theorem \ref{HomogLambZero}}
\label{sec:UniformEstimatesAndInterchangingTheLimits}

Combining Proposition \ref{pro:pseqweaklimit} and \ref{pro:formalLimit} we see that 
$(1-\gamma L_r)^M G_0 f$ converges weakly to the solution of \eqref{eq:homoPDE} if we take 
the limits in the order $r \to 0$ followed by $M \to \infty$.
In order to prove Theorem \ref{HomogLambZero}, it remains interchange the order of taking the limits.
For this purpose, we will prove that the speed of convergence of $(1-\gamma L_{r})^M G_0 f$ to $u_r$ in $\dot{H}^1_\loc(\IR^3)$ 
as $M$ tends to infinity is uniform in $r$ for fixed $\mu$.

Corresponding to Lemma \ref{lem:LCoercive}, we have the following lemma. It implies that 
the sequence $(1-\gamma L_r)^M G_0 f$ is close to zero boundary conditions in the particles in any fixed bounded region
uniformly in $r$ as $M \to \infty$.

\begin{lemma}
	\label{lem:LCoerciveInCompacta}
	Let $ u \in \dot{H}_0^1(\IR^3 \backslash K_r)^\perp $ and $R > 0$, we define $v \in \dot{H}^1(\IR^3)$ to be the solution
	to
	\begin{align}
		-\Delta v &= 0 \quad \text{in} ~ \IR^3 \backslash (K_r \cap \overline{B_R(0)}), \\
		v &= u  \quad \text{in} ~ K_r \cap \overline{B_R(0)}.
	\end{align}
	Then,
	\[
		 (L_r u,u)_{\dot{H}^1(\IR^3)} \geq c e^{-R} \| v\|_{\dot{H}^1(\IR^3)}^2,
	\]
	where $c>0$ is a universal constant.
\end{lemma}

\begin{proof}
	Let $ \eta_x \in C_c^\infty (B_{2r}(x)) $ such that 
	$ \eta_x = 1 $ in 	$ B_{r}(x) $ and $ |\nabla \eta_x | \leq \frac{C}{r} $. 
	Now, we observe that for all $ w \in \dot{H}^1(\IR^3) $
	\[
		\| w \|_{L^2(B_{2r}(x))} \leq \| w \|_{L^6(B_{2r}(x))} \|1\|_{L^3(B_{2r}(x))} 
		\leq  C r \| \nabla w \|_{L^2(\IR^3)},
	\]
	and hence,
	\begin{equation}
		\| \eta_x w \|_{\dot{H}^1(\IR^3)} \leq \|w\|_{\dot{H}^1(\IR^3)} + \frac{C}{r} \| w \|_{L^2(B_{2r}(x))} 
		\leq C \| w \|_{\dot{H}^1(\IR^3)}.
	\end{equation}
	
	On the other hand, by the variational form of the equation for $ v $, we know that $v$ is the function of minimal norm 
	in the set $ X_v := \{ w \in \dot{H}^1(\IR^3) \colon w = v ~ \text{in} ~ K_r \cap \overline{B_R} \}$. 
	Clearly, $\sum_{x \in B_{R+r}} \eta_x Q_x v \in X_u$, and hence,
	\begin{align}
		\langle L_r v , v \rangle &= \sum_x e^{-|x|} \|Q_x v \|_{\dot{H}^1(\IR^3)}^2 \\
		&\geq c e^{-R} \sum_{x \in B_{R+r}} \|\eta_x Q_x v \|_{\dot{H}^1(\IR^3)}^2 \\
		&= c e^{-R} \|\sum_{x \in B_{R+r}} \eta_x Q_x v \|_{\dot{H}^1(\IR^3)}^2 \\
		&\geq c e^{-R} \|v\|^2_{\dot{H}^1(\IR^3)}. \qedhere
	\end{align}
\end{proof}

The next Lemma is needed to ensure that the values of $(1-\gamma L)^M G_0 f$ in a fixed bounded region 
is very little affected by particles far away from this region.

\begin{lemma}
	\label{lem:decayQuasiHarmonic}
	For all $\mu > 0$, there exists a nonincreasing function 
	$ e_\mu \colon \IR_+ \to \IR_+$ with $ \lim_{s \to \infty} e_\mu(s) = 0$ that has the following property.
	For all $0 \leq \rho \leq R$, all $w \in \dot{H}_0^1(\IR^3 \backslash K_r)^\perp $  with $w = 0$ in 
	$ K_r \cap B_R(0)$ satisfy 
	\[
		\| \nabla w\|_{L^2(B_\rho(0))} \leq e_\mu(R-\rho) \| \nabla w \|_{L^2(\IR^3)},
	\]
	if $r$ is sufficiently small.
\end{lemma}

\begin{proof}
	The proof uses the classical Widman's hole filling technique (see e.g. \cite{Gia83}).
	Fix a particle configuration with capacity $\mu $ and $d<1/(2\sqrt{3})$, and fix $R$, $\rho$, and $w$ according to the assumptions. 
	For $1 \leq s \leq R - 1$, we define $\eta_s \in C_c^\infty (B_{1+s}(0))$ such that $\eta_s = 1 $ in $B_s(0)$, $|\eta_s| \leq 1$,
	and $ |\nabla \eta_s | \leq C $. 
	We use $\eta^2 w$ as a test function in the weak form of the equation $w$ satisfies, namely,
	\begin{align}
		-\Delta w &= 0 \quad \text{in} ~ \IR^3 \backslash K_r \\
		w &= 0 \quad \text{in} ~ K_r \cap B_{s+1}.	
	\end{align}
	This yields
	\[
		0 = \int_{B_{s+1}} \nabla w \nabla (\eta^2 w) \dd x = \int_{B_{s+1}} (\eta \nabla w)^2 + 2 \eta \nabla w \nabla \eta w  \dd x.
	\]
	Using the Cauchy-Schwartz inequality and the Poincaré inequality in the annulus $B_{s+1} \backslash B_s $, provided by
	Lemma \ref{lem:poincareAnnulus}, we deduce 
	\[
		\|\nabla w\|^2_{L^2(B_s)} \leq \| \eta \nabla w\|_{L^2(B_s)} \leq C \| w \|^2_{L^2(B_{s+1} \backslash B_s)}
		 \leq C (1+\mu^{-1}) \| \nabla w \|^2_{L^2(B_{s+1} \backslash B_s)}.
	\]
	Let us denote $ a_k := \|\nabla w\|^2_{L^2(B_{\rho+k})} $. Then, the above estimate implies for all 
	$ k $ such that $ \rho + k \leq R - 1$
	\[
		a_k \leq C(1 + \mu^{-1}) (a_{k+1} - a_k).
	\]
	Therefore,
	\[
		a_k \leq \frac{C(1 + \mu^{-1})}{C(1+\mu^{-1}) +1} a_{k+1} =: \lambda_\mu a_{k+1},
	\]
	and $\lambda_\mu < 1$.
	By iterating up to $ n = \lfloor R-\rho-1 \rfloor$, we conclude
	\[
		\|\nabla w\|^2_{L^2(B_\rho)} \leq \lambda_\mu^n \|\nabla w\|^2_{L^2(\IR^3)}.
	\]
	This is the desired estimate with $e_\mu(s) := \lambda_\mu^\frac{{\lfloor s-1 \rfloor}}{2} $ (for $s \geq 1$ and 
	$ e_\mu = 1$ otherwise).
\end{proof}

\begin{remark}
	\label{rem:exponentialDecay}
	As seen in the proof, the decay $e_\mu$ is exponential. This can be interpreted as a screening effect due to the presence 
	of the particles. This effect can be exploited to prove homogenization results also for sources $f \in L^{\infty}(\IR^3)$
	(cf. \cite{NV1}, \cite{NV}).
\end{remark}

\begin{lemma}
	\label{lem:poincareAnnulus}
	Let $s \geq 1$ and $d < 1/(2\sqrt{3}) $.
	Then, for all $u \in \dot{H}^1_0(\IR^3 \backslash K_r)$,
	\[
		\| u\|^2_{L^2(B_{s+1} \backslash B_s)} 
		\leq \frac{2\sqrt{3}}{\mu} \| \nabla u\|^2_{L^2(B_{s+1} \backslash B_s)}.
	\]
\end{lemma}

\begin{proof}
As the Poincaré inequality for the whole space $\IR^3$, Corollary \ref{cor:PoincarePerforated}, the basically follows from
the estimate
\[
	\|u\|_{L^2(B_R(z))}^2 	\leq \frac{R^3}{\rho} \|\nabla u\|^2_{L^2(B_R(z))}
\]
if $u=0$ in $B_\rho(z)$, which is the statement of Lemma \ref{lem:poincare}. 
However, there are certain technical issues due tothe nonconvexity of the annulus. 

	Let us denote the annulus $B_{s+1}(0) \backslash B_s(0) $ by $A_s$.
	First observe that Lemma \ref{lem:poincare} remains true if we replace $B_R(z)$ by any $\Omega \subset B_R(z)$ that
	is star-shaped with respect to $z$. The reason is that we only integrated over line segments with endpoint $z$.
	Therefore, the assertion follows, once we have shown that there exists a covering
	\[
		A_s \subset \cup_{x \in \Gamma_r} B_{R_x}(x),
	\]
	such that for all $x \in \Gamma_r$ the set $ A_s \cap B_{R_x}(x)$ is star-shaped with respect to 
	$x$ and $R_x \leq 2 \sqrt{3}d$. Equivalently, for every point $y$ in the annulus, we have to find $x \in A_s \cap \Gamma_r$
	and $R_x \leq 2\sqrt{3}d$ such that $y \in B_{R_x}(x)$ and $ A_s \cap B_{R_x}(x)$ is star-shaped with respect to 
	$x$.
	
	For $y \in A_s$, there exists a ball $B_{2\sqrt{3}d}(z_1) \subset A_s$ that contains $y$, since $d < 1/(2\sqrt{3})$.
	In this ball we find $B_{\sqrt{3}d}(z_2) \subset B_{2 \sqrt{3} d}(z_1)$ with distance $\sqrt{3}d$ from the inner boundary of the annulus, i.e.,
	$\dist\{\partial B_s(0), B_{\sqrt{3}d}(z_2)\} \geq \sqrt{3}d$. By definition of the particle configuration, 
	there exists $x \in B_{\sqrt{3}d}(z_2) \cap \Gamma_r$.
	Moreover, by construction, we have $ y \in B_{2\sqrt{3}d}(x)$. 
	
	Finally, we prove that $A_s \cap B_{2\sqrt{3}d}(x)$ is
	star-shaped with respect to $x$. Clearly, the only problem can occur by removing the inner ball $B_s(0)$ from
	$B_{2\sqrt{3}d}(x)$. If $B_{2\sqrt{3}d}(x) \cap B_s(0)$ is empty, then, we are done. 
	If not, $A_s \cap B_{2\sqrt{3}d}(x)$ is star-shaped with respect to $x$
	if, for any $z \in \partial B_s \cap \partial B_{2\sqrt{3}d}(x)$, the line segment $l$
	from $z$ to $x$ is disjoint from $B_s(0)$. Clearly, it is equivalent to check that $l$ has smaller length than
	any line segment $t$ from $x$ to some $w \in \partial B_s(0)$ that is tangential to $\partial B_s(0)$.
	Since $\dist \{\partial B_s(0), B_{\sqrt{3}d}(z_2)\} \geq \sqrt{3}d$ and $s \geq 1$, it follows 
	\[
		|t|^2 \geq (s+\sqrt{3}d)^2-s^2 = 2s\sqrt{3}d + 3 d^2 \geq 2\sqrt{3}d \geq 12 d^2 = |l|^2.
	\]
	This finishes the proof.
\end{proof}

\begin{proposition}
	\label{pro:pSolByScattering}
Let $f\in \dot{H}^{-1}\left(  \mathbb{R}^{3}\right).$ For all $0 < \mu_1 \leq \mu_2 < \infty$, 
there exists a $\gamma >0$ depending only on $\mu_1$ and $\mu_2$ such that the sequence
\begin{equation}
\lim_{N\rightarrow\infty}\bigg(  1-\gamma\sum_{j}e^{-\left\vert
x_{j}\right\vert }Q_{j}\bigg)  ^{N} G_0 f\ %
\end{equation}
converges to the solution of \eqref{S1E1} uniformly in  $\dot{H}^{1}_\loc(\IR^3)$
for all particle configuration with capacity $ \mu_1 \leq \mu \leq \mu_2$ and sufficiently small $r$.
\end{proposition}

\begin{proof}
As in the proof of Theorem \ref{ConvWholeSpace}, we choose $\gamma \leq 1/\|L_r\|$. Lemma \ref{lem:Lbounded} ensures that this is possible such that $\gamma$ depends only on $\mu$ if $r$ is sufficiently small.

Let $\rho >0 $, $ \eps>0$, and 
$u := G_0 f \in \dot{H}^1(\IR^3)$. Since $\ker (L_r) = \dot{H}_0^1(\IR^3 \backslash K_r)$, it suffices to consider
 $u \in \dot{H}_0^1(\IR^3 \backslash K_r)^\perp$. 
Define $u_M := (1-\gamma L_r)^M u$. 

Then, we know from Proposition \ref{pro:nonuniformAbstractProjection}
\begin{align}
	\|(1-\gamma L_r)u\|_{\dot{H}^1(\IR^3)}^2 &= \|u\|_{\Hdot}^2 
	- 2(\gamma L_r u,u)_{\Hdot} +\|\gamma L_r u\|^2_{\Hdot} \\
	&\leq \|u\|_{\Hdot}^2 - \gamma (L_r u ,u)_{\Hdot}.
\end{align}
Iterating and using monotonicity of  $(L_r u_M, u_M)_{\Hdot}$, 
which follows from the estimate \eqref{eq:monotonPos} in Proposition 
\ref{pro:nonuniformAbstractProjection}, yields
\[
	0 \leq \|u_{M+1}\|^2_{\Hdot} \leq \|u\|_{\Hdot}^2 - (M+1) \gamma (L_r u_M, u_M)_{\Hdot}.
\]
Thus,
\[
	 (L_r u_M, u_M)_{\Hdot} \leq \frac{1}{(M+1) \gamma} \|u\|_{\Hdot}^2.
\]
Define $v_M \in \dot{H}^1(\IR^3)$ to be the solution	to
	\begin{align}
		-\Delta v_M &= 0 \quad \text{in} ~ \IR^3 \backslash (K_r \cap \overline{B_R}), \\
		v_M &= u_M  \quad \text{in} ~ K_r \cap \overline{B_R},
	\end{align}
	and $w_M := u_M -v_M$.
	Then, Lemma \ref{lem:decayQuasiHarmonic} implies for all $R > \rho$
	\begin{align}
		\| \nabla w_M\|_{L^2(B_\rho(0))} \leq e_\mu(R-\rho) \| w_M \|_{\Hdot} 
		&\leq e_\mu(R-\rho) (\|  u_M \|_{\Hdot} +\|  v_M \|_{\Hdot}) \\
		& \leq e_\mu(R-\rho) (\|  u \|_{\Hdot} +\|  v_M \|_{\Hdot}),
	\end{align}
		and it is possible to choose $R$ large enough such that $ e_\mu(R-\rho) < \frac{\eps}{3}$.
		On the other hand, by Lemma \ref{lem:LCoerciveInCompacta}, we have
	\[
		 c e^{-R} \| v_M\|_{\dot{H}^1(\IR^3)}^2 \leq (L_r u_M,u_M)_{\dot{H}^1(\IR^3)} \leq \frac{1}{(M+1) \gamma} \|u\|_{\Hdot}^2.
	\]
	Therefore, choosing $M_0$ large enough yields for all $M \geq M_0$
	\[
		\| v_M\|_{\dot{H}^1(\IR^3)} < \frac{\eps}{3} \|u\|_{\Hdot}.
	\]
	By combining the estimates for $v_M$ and $w_M$, we conclude (assuming without restriction $\eps \leq 3$)
	\[
		\| \nabla u_M \|_{L^2(B_\rho(0))} < \eps \|u\|_{\dot{H}^1(\IR^3)} = \eps \|f\|_{\dot{H}^{-1}(\IR^3)}. \qedhere
	\]
\end{proof}

\begin{proof}[Proof of Theorem \ref{HomogLambZero}]
	We first prove that $u$ converges weakly in $\Hdot(\IR^3)$ for all sources $f \in \dot{H}^{-1}(\IR^3)$.
	Since the sequence is bounded, it suffices to consider test functions in $C_c^\infty(\IR^3)$.
	Let $ \varphi \in C_c^\infty(\IR^3) $ and choose $R>0$ such that $\supp \varphi \subset B_R(0)$.
	Further, let $\gamma < \gamma_0 $ from Proposition \ref{pro:pSolByScattering} and denote by $S_M$ the 
	corresponding pointwise	weak limit of $(1-\gamma L_{r})^M$ from Proposition \ref{pro:pseqweaklimit}.
	Then, for all $M>0$,
	\begin{align}
		|(u_r - u,\varphi)_{\Hdot}| &\leq |(S_M G_0 f - u,\varphi)_{\Hdot}| 
		 + 
		|(1- \gamma L_r)^M G_0 f - S_M G_0 f ,\varphi)_{\Hdot}| \\
		{} &+ |(u_r - (1- \gamma L_r)^M G_0 f,\varphi)_{\Hdot}|. 
	\end{align}
	The third term on the right hand side is estimated by 
	\[
		\|\nabla(u_r - (1- \gamma L_r)^M G_0 f)\|_{L^2(B_R)} \|\varphi\|_{\Hdot}.
	\]
	By choosing $M$ sufficiently large, since Proposition \ref{pro:pSolByScattering} ensures that this term becomes small
	independently of $r$. On the other hand, also the first term becomes small by choosing $M$ large,
	and the second term vanishes in the limit $r \to \infty$.
	
	Weak convergence in $\Hdot(\IR^3)$ is equivalent to weak convergence in $L^2(\IR^3)$ 
	of the gradients. However, due to Corollary \ref{cor:Existence}, 
	the sequence $u_r$ is uniformly bounded in $H^1(\IR^3)$.
	Therefore,  we can extract subsequences that converge weakly in $H^1(\IR^3)$. 
	Since their weak limit is uniquely determined by 
	the weak limit of their gradients, the whole sequence converges weakly in $H^1(\IR^3)$.
	
	The result for $f \in H^{-1}(\IR^3)$ follows from density of $\dot{H}^{-1}(\IR^3) $ in $H^{-1}(\IR^3)$ 
	using again that the solution operators for Problem \eqref{eq:poissonPerforated} are uniformly bounded.
\end{proof}

\subsection{Proof of Proposition \ref{pro:pseqweaklimit}}

\begin{lemma}
	\label{lem:pMonopapprox}
	The following holds for the operators defined in Definition \ref{def:monopole} and \ref{def:A_beta}.
	\begin{enumerate}[(i)]	
	\item There exists a constant $C$ such that, for all $x \in \IR^3$  and $u \in \Hdot$,
	
	\begin{equation}
		\label{eq:pMonopapprox}
		\|(T_x- M_{x})u\|_{\dot{H}^{-1}(\IR^3)} \leq C \| \nabla u \|_{L^2(B_x)}.
	\end{equation} 
	\item For fixed $ \mu $ and $r \to 0$, we have for all $ u \in \dot{H}^{1}(\IR^3) $, all $ n \in \IN$, and all 
	$\beta \in \IN_\ast^n$,
	\begin{equation}
		\label{eq:pMonopgood}
		\|M_\beta^{(r)} - A_\beta^{(r)})u\|_{\Hdot} \to 0 \qquad \text{as} \quad k \to \infty.
	\end{equation}
\end{enumerate}
\end{lemma}

For the proof we need the following lemma.
	
\begin{lemma}
	\label{lem:extest}
	For $ r>0 $ and $ x \in \IR^3$,  let $ H_r := \{ u \in  H^1(B_r(x)) \colon \int_{B_r(x)} u = 0 \} $. 
	Then, for all $ r >0 $, there exists an extension operator $ E_r \colon H_r \to H^1_0(B_{2r}(x)) $ such that
	\begin{equation}
		\label{eq:extest}
		\| \nabla E_r u \|_{L^2(B_{2r}(x))} \leq C \| \nabla u \|_{L^2(B_r(x))} \qquad \text{for all} \quad u \in H_r,
	\end{equation}
	where the constant $ C $ is independent of $ r $.
\end{lemma}
\begin{proof}
	For $ r = 1 $ let $ E_1 \colon H^1(B_1(x)) \to H^1_0(B_2(x)) $ be a continuous extension operator.
	Then, by the Poincaré inequality on $ H_1 $, we get for all $ u \in H_1 $
	\begin{equation}
		\| \nabla E_1u \|_{L^2(B_{2}(x))} \leq \| E_1u \|_{H^1(B_{2}(x))} \leq C \| u \|_{H^1(B_1(x))}
		\leq C \| \nabla u \|_{L^2(B_1(x))}.
	\label{eq:extest1}
	\end{equation}
	The assertion for general $ r > 0 $ follows from scaling by defining 
	$ (E_r u)(x) := (E_1u_r)({\frac{x}{r}}) $, where $ u_s(x) := u(sx) $.
\end{proof}

\begin{proof}[Proof of Lemma  \ref{lem:pMonopapprox}]
	Let $u \in \dot{H}^1(\IR^3)$.
	First, we observe by a straightforward calculation that
	\begin{equation}
			(\tilde{Q}_{x} u)(y) = 
		\begin{cases}
			(u)_{x}, & \text{if} \quad x \in B_{x},\\
			(u)_{x} \displaystyle\frac{r}{|y-x|}, & \text{otherwise}.
		\end{cases}
	\end{equation}
	Now, we use again that $G_0$ is an isometry and that $ Q_x = G_0 T_{x} $ is the orthogonal projection to the subspace
	\[
		\dot{H}_0^1(\IR^3 \backslash \overline{B_{x_1}})^\perp = 
		\left\{ u \in H^1(\IR^3) \colon -\Delta u = 0 ~ \text{in} ~ \IR^3 \backslash ( \overline{B_{x}}) \right\}.
	\]
	Therefore, we can characterize $Q_x u$ as the function $v \in \dot{H}^1(\IR^3)$ that solves
	\begin{align}
		-\Delta v &= 0 \quad \text{in} ~ \IR^3 \backslash \overline{B}_x, \\
		v &= u \quad \text{in} ~ \overline{B}_x.
	\end{align}
	Hence, $v$ is the function of minimal	norm that coincides with $u$ inside the ball $B_x$.
	Clearly, $\tilde{Q}_{x} u \in \dot{H}_0^1(\IR^3 \backslash \overline{B_{x}})^\perp$, 
	and thus, $Q_{x} \tilde{Q}_{x} = \tilde{Q}_x$.
	Therefore,	
	\[
		(Q_{x}  - \tilde{Q}_{x})u = Q_x (u - \tilde{Q}_{x} u).
	\]
	Since $ \tilde{Q}_{x} u = (u)_{x} $ in $ B_{x} $, 
	we can use the extension operator  $ E_r $ from Lemma \ref{lem:extest} (since, by the Rellich embedding theorem,
	the restriction of a $\Hdot$ function to a ball is a $H^1$ function in that ball)
		and estimate
	\begin{align}
		\|(Q_{x} - \tilde{Q}_{x})u\|_{\dot{H}^1(\IR^3)} 
		&\leq \| E_r ((u - \tilde{Q}_{x} u) \left. \! \! \right|_{B_{x}})\|_{\dot{H}^1(\IR^3)} \\
		&= \| \nabla E_r (( u - (u)_{x}) \left. \! \! \right|_{B_{x}})\|_{L^2(\IR^3)} \\
		& \leq C \| \nabla u \|_{L^2(B_{x})}.
	\end{align}
	This concludes the proof of assertion (i).

	Observe that $ M_{x} u $ satisfies 
	$ \supp (M_{x} u) \subset \partial B_{x}$. It can easily be seen that Lemma \ref{lem:pLocH^-1} still holds true
	when replacing the cutoff $e^{-|x|}$ by $e^{-j|x|}$ for any $j \in \IN_\ast$. Therefore, we get for $ n = 1$
	\[
		\|(M^{(r}_\beta - A_\beta^{(r)})u\|_{\dot{H}^1(\IR^3)}^2 
		\leq (1+C\mu) \sum_{x} e^{-\beta |x_1|}\|(Q_{x} - \tilde{Q}_{x})u  \|_{\dot{H}^1(\IR^3)}^2.
	\]
	Hence, the convergence \eqref{eq:pMonopgood} for $ n = 1$ follows directly from the estimate \eqref{eq:pMonopapprox}, 
	since, for fixed capacity,	the volume of the particles inside a fixed bounded domain 
	tends to zero as $ r $ tends to zero.
	
	For $n \geq 2$, we first argue that it suffices to prove the assertion for functions $u$ in the dense set
	$H^1(\IR^3) \subset \dot{H}^{1} (\IR^3) $. 
	Indeed, this follows once we have shown that, for any $\beta$, both $ A_\beta^{(r)}$ and $M_\beta^{(r)}$ are bounded, uniformly	for all particle configurations with the same capacity $\mu$. For $ A_\beta^{(r)}$, this is the second
	statement of Lemma \ref{lem:powersOfL}. 
	To estimate $M_\beta^{(r)}$, we consider first $\beta = 1$. 
	By part (i), we have for all $u \in \Hdot$
	\[
		\sum_x \|(\tilde{Q}_x - Q_x)u\|^2_{\Hdot} \leq C \|u\|_{\Hdot}^2.
	\]
	Note that we can use 
	Lemma \ref{lem:pLocH^-1} to take the sum out of the norm in the definition of $M_1^{(r)}$,
	since $M_x u$ is supported on $\partial B_x$ for all $u \in \Hdot$. 
	Using additionally the bound for $L_r$ from Lemma \ref{lem:Lbounded}, we estimate
	\begin{align}
		\|M_1^{(r)} u\|_{\Hdot}^2 &= \| \sum_x e^{-|x|} \tilde{Q_x}u\|_{\Hdot}^2 \\
		&\leq    \|\sum_x e^{-|x|} (\tilde{Q_x}- Q_x)u\|_{\Hdot}^2 +  \|L_r u\|_{\Hdot}^2 \\
		&\leq (1 + C \mu) \sum_x \|e^{-\beta_1} (\tilde{Q_x}- Q_x)u\|_{\Hdot}^2 + (1+ C\mu)^2 \|u\|_{\Hdot}^2 \\
		&\leq C(1+\mu)^2 \| u\|_{\Hdot}^2.
	\end{align}
	For general $n \in \IN_\ast$ and $\beta \in \IN_\ast^n$, one can argue as in the proof of Lemma \ref{lem:powersOfL}
	 by taking powers of $M_1^{(r)}$ in order to deduce $\|M_\beta^{(r)}\| \leq (C(1+\mu))^n$.
	Indeed, the only ingredient for the proof of the formula, which we derived in Lemma \ref{lem:powersOfL}
	for $(L_r)^n$, was that $Q_x$ is a projection. We did not use orthogonality.
	Since $\tilde{Q}$ is a projection as well, the analogous version of that formula holds for the powers of $M_1^{(r)}$.
	
	The general assertion now follows by induction. For $ n=2 $, we have
	\begin{equation}
	\begin{aligned}
		 \| (M^{(r)}_\beta - A_\beta^{(r)})u \|^2_{\dot{H}^1(\IR^3)} 
			&\leq \Big \| \sum_{x_1} \sum_{{x_2} \neq {x_1}} e^{-\beta_1|x_1|} e^{-\beta_2|x_2|} 
			Q_{x_1}(Q_{x_2} - \tilde{Q}_{x_2})  u \Big \|^2_{\dot{H}^1(\IR^3)} \\
			\quad{} &+ \Big \| \sum_{x_1} \sum_{{x_2} \neq {x_1}} e^{-\beta_1|x_1|} e^{-\beta_2|x_2|} 
			(\tilde{Q}_{x_1} - Q_{x_1}) \tilde{Q}_{x_2} u \Big \|^2_{\dot{H}^1(\IR^3)}.
	\end{aligned}
			\label{eq:pSplitMM}
	\end{equation}
	To further estimate the first term on the right hand side, we use that $ \sum_{x_1} e^{-\beta_1|x_1|}  Q_{x_1}$ 
	is a bounded operator.
	Together with part (i) and using again $Q_x Q_x = Q_x$ and $Q_x \tilde{Q}_x = \tilde{Q}_x$,
	we get (with a constant that depends on $\mu$)
	\begin{equation}
		\begin{aligned}
			&\Big \| \sum_{x_1} \sum_{{x_2} \neq {x_1}} e^{-\beta_1|x_1|} e^{-\beta_2|x_2|} 
			Q_{x_1}(Q_{x_2} - \tilde{Q}_{x_2})  u \Big \|_{\dot{H}^1(\IR^3)} \\
			&\leq  \Big \| \sum_{x_1} e^{-\beta_1|x_1|} Q_{x_1} \sum_{{x_2}}  e^{-\beta_2|x_2|}
			(Q_{x_2} - \tilde{Q}_{x_2}) u \Big \|_{\dot{H}^1(\IR^3)} 
			+ \Big \| \sum_{x_1} e^{-\beta_1|x_1|} Q_{x_1} (Q_{x_1} - \tilde{Q}_{x_1}) u \Big \|_{\dot{H}^1(\IR^3)}\\
			&\leq C \Big \| \sum_{{x_2}}  e^{-\beta_2|x_2|}(Q_{x_2} - \tilde{Q}_{x_2}) u \Big \|_{\dot{H}^1(\IR^3)} 
			+ C \Big  \|\sum_{x_1}e^{-\beta_1|x_1|}(Q_{x_1} - \tilde{Q}_{x_1}) u \Big \|_{\dot{H}^1(\IR^3)} \\
			& \leq C \Big \| \sum_{{x_2}}  e^{-\beta_2|x_2|}(Q_{x_2} - \tilde{Q}_{x_2}) u \Big \|_{\dot{H}^1(\IR^3)} \to 0.
		\end{aligned}
	\end{equation}
	For the second term on the right hand side of \eqref{eq:pSplitMM}, recall
	\[
		(M_{x_1} - T_{x_1}) v = 0 \quad \text{in} ~ \IR^3\backslash \overline{B_{x_1}}
	\] 
	for all $ v \in \dot{H}^1(\IR^3) $.
	Hence, for $u \in H^1(\IR^3)$, we can use Lemma \ref{lem:pLocH^-1} 
	to take out the sum in $ {x_1} $, and we use the estimate for the uniform charge density approximation from part (i),
	\begin{equation}
		\label{eq:pInductionStart}
		\begin{aligned}
		\Big \|  \sum_{x_1} \sum_{{x_2} \neq {x_1}} e^{-\beta_1|x_1|} e^{-\beta_2|x_2|}
		(\tilde{Q}_{x_1} - Q_{x_1}) \tilde{Q}_{x_2} u
		\Big \|^2_{\dot{H}^1(\IR^3)} 
		\leq \sum_{x_1} e^{-\beta_1|x_1|} \Big \|\! \sum_{{x_2} \neq {x_1}}  e^{-\beta_2|x_2|}\nabla \tilde{Q}_{x_2} u \Big 
		\|^2_{L^2(B_{x_1})}.
		\end{aligned}
	\end{equation}
	Inserting the definition of $ \tilde{Q}_{x_2} $, expanding the square of the sum over 
	$ {x_2} $, and estimating the integral yields
	\begin{equation}
			\label{eq:pr^5sum}
	\begin{aligned}
		&\sum_{x_1} e^{-\beta_1|x_1|} \Big \| \! \sum_{{x_2} \neq {x_1}} e^{-\beta_2|x_2|}\nabla \tilde{Q}_{x_2} u \Big 
		\|^2_{L^2(B_{x_1})} \\
		&\leq C \sum_{x_1} \sum_{{x_2} \neq {x_1}} \sum_{{x_3} \neq {x_1}} e^{-\beta_1|x_1|}  
		r^5  
	\frac{e^{-2\beta_2|x_2|}(u)_{x_2}}{|x_1-x_2|^2} \frac{e^{-2\beta_3|x_3|} (u)_{x_3} }{|x_1-x_3|^2}.
	\end{aligned}
	\end{equation}
	Consider the off-diagonal terms first, i.e., $ {x_2} \neq {x_3} $. We estimate 
	\[
		e^{-\beta_1|x_1|} e^{-2\beta_2|x_2|} e^{-2\beta_3|x_3|} \leq e^{-\frac{|x_1-x_2|}{2}} e^{-\frac{|x_1-x_3|}{2}}
	\] 
	and use $ r = (4 \pi)^{-1} \mu d^3 $ to bound	the sum over $ {x_1} $ by an integral,
	\[
		\sum_{x_1} r \frac{e^{-\frac{|x_1-x_2|}{2}}}{|x_1-x_2|^2} \frac{e^{-\frac{|x_1-x_3|}{2}}}{|x_1-x_3|^2}
		\leq C \mu\int_{\IR^3} \frac{e^{-\frac{|y-x_2|}{2}}}{|y-x_2|^2} \frac{e^{-\frac{|y-x_3|}{2}}}{|y-x_3|^2} \dd y,
	\]
	for all $ {x_2} \neq {x_1}$, ${x_3} \neq {x_1}$.
	To estimate the integral, we denote $ z = x_2 - x_3 \neq 0$ for the moment and split the integral to get
	\[
	\begin{aligned}
		\int_{\IR^3} \frac{e^{-\frac{|y|}{2}}}{|y|^2} \frac{e^{-\frac{|y-t|}{2}}}{|y-t|^2} \dd y 
		\leq \int_{\IR^3 \backslash B_{|z|/2}(0)} \frac{4 e^{-\frac{|z|}{4}}}{|z|^2} \frac{e^{-\frac{|y-z|}{2}}}{|y-z|^2}
		+ \int_{B_{|z|/2}(0)} \frac{e^{-|y|}}{|y|^2} \frac{4 e^{-\frac{|z|}{4}}}{|z|^2} 
		 \leq C \frac{e^{-\frac{|z|}{4}}}{|z|^2}.
	\end{aligned}
	\]
	Hence, using $ (u)_{x_2} (u)_{x_3}  \leq \frac{1}{2} ((u)_{x_2}^2 +  (u)_{x_3}^2) $ and symmetry, we deduce 
	\[
	\begin{aligned}
		\sum_{x_1} \sum_{{x_2} \neq {x_1}}  \sum_{\substack{{x_3} \neq {x_1} \\ {x_3} \neq {x_2}}} 
		r^5 (u)_{x_2} (u)_{x_3} 
		\frac{e^{-\frac{|x_1-x_2|}{2}}}{|x_1-x_2|^2} \frac{e^{-\frac{|x_1-x_3|}{2}}}{|x_1-x_3|^2} 
		&\leq \sum_{x_2} \sum_{{x_3} \neq {x_2}} C \mu r^4 (u)_{x_2} (u)_{x_3} 
		\frac{e^{-\frac{|x_2 - x_3|}{4}}}{|x_2 - x_3|^2} \\
		& \leq \sum_{x_2} C \mu^2 r^3 (u)_{x_2}^2 \int_{\IR^3} \frac{e^{-\frac{|x_2 - y|}{4}}}{|x_2 - y|^2} \dd y \\
		& \leq C \mu^2 \sum_{x_2} r^3 \left(\fint_{B_{x_2}} u(y) \dd y \right)^2 \\
		& \leq C \mu^2 \| u \|^2_{L^2(\cup_{x_2} B_{x_2})} \to 0,
	\end{aligned}
	\]
	where we finally used $u \in H^1(\IR^3)$ in order to have control of the $L^2$-norm. 
	It remains to bound the diagonal terms in \eqref{eq:pr^5sum}. For those, we use the estimate
	\[
		r \sum_{{x_1} \neq {x_2}} \frac{e^{-\frac{|x_1-x_2|}{2}}}{|x_1-x_2|^4} 
		\leq C \mu \int_{\IR^3 \backslash B_d(0)} \frac{e^{-\frac{|y|}{2}}}{|y|^4} \dd y \leq C \mu d^{-1}.
	\]
	Hence, 
	\[
		\sum_{x_1} \sum_{{x_2} \neq {x_1}} r^5 (u)_{x_2}^2\frac{e^{-\frac{|x_1-x_2|}{2}}}{|x_1-x_2|^4} 
		\leq C \mu r d^{-1} \| u \|^2_{L^2(\cup_{x_2} B_{x_2})} \to 0.
	\]
	
	For $ n \geq 3 $, one does exactly the same thing. For $ n = 3 $, instead of 
	$ \| u \|^2_{L^2(\cup_{x_2} B_{x_2})} $, one ends up with
	\begin{equation}
		\begin{aligned}
			\sum_{x_2} \Big \| \! \sum_{{x_3} \neq {x_2}} \tilde{Q}_{x_3} u \Big \|^2_{L^2(B_{x_2})} 
		\end{aligned}
	\end{equation}
	and this can be estimated exactly as the right hand side of Equation \eqref{eq:pInductionStart}.
	The only difference is that the gradient is not there, but, due to the exponential cutoff, this does not matter.
	Thus, for $n \geq 3$, the assertion follows by induction.
	\end{proof}

\begin{proof}[Proof of Proposition \ref{pro:pseqweaklimit}]
	By Lemma \ref{lem:powersOfL} it suffices to prove 
	\begin{equation}
		\label{eq:sepseqweaklimit}
		A_\beta^{(r)} u \rightharpoonup \Bigg( \prod_{j=1}^n \mu G_0 J A^{\beta_j} \Bigg) u \quad \text{in} ~ \Hdot,
	\end{equation}
	for all $ u \in \dot{H}^{1}(\IR^3) $, all $ n \in \IN_\ast $, and all $ \beta \in \IN_\ast^n$.

	Since $G_0$ is an isometry, for $ n = 1$, it suffices to show 
	\[	
		\sum_x e^{-\beta_1|x|} T_x u \rightharpoonup  \mu J A^{\beta_1} u \quad \text{in} ~ \dot{H}^{-1}(\IR^3)
	\]
	for all $ u \in \dot{H}^1(\IR^3) $ and analogously for $n \geq 2$.
	Since by Lemma \ref{lem:powersOfL}, 
	we have a uniform bound on $A_\beta^{(r)}$, it suffices to prove the assertion for the 
	dense subset $ \dot{H}^1(\IR^3) \cap C^1(\IR^3)$. Lemma \ref{lem:pMonopapprox} implies 
	that we can replace all the operators
	$T_x$ by $ M_x $. Moreover, it suffices to consider test function from the dense set $C_c^\infty(\IR^3)$. 
	Let $ u \in \dot{H}^1(\IR^3) \cap C^1(\IR^3) $ and $ \varphi \in C_c^\infty(\IR^3) $. Then, we estimate
	\begin{align}
		|\langle \varphi , M_x u \rangle - 4 \pi r u(x) \varphi(x)| 
		&\leq \frac{1}{r} \int_{\partial B_x} | (u)_x \varphi(y) - u(x) \varphi(x)| \dd y \\
		& \leq C r^2 \|u\|_{C^1(\IR^3)} \|\varphi\|_{C^1(\IR^3)}.       
	\end{align}
	On the other hand, defining $ q_x $ to be the cube centered at $x$ with edges of length $ d $ parallel to the
	coordinate axes, we find
	\begin{align}
		\bigg| \mu \int_{q_x} e^{-\beta_1 |y|}u(y) \varphi(y) \dd y -  r e^{-\beta_1 |x|}u(x) \varphi(x) \bigg| 
		&\leq \mu\int_{q_x} | e^{-\beta_1 |y|}u(y) \varphi(y) - e^{-\beta_1 |x|}u(x) \varphi(x)| \dd y \\
		& \leq C r d \|e^{-\beta_1 |\cdot|}u\|_{C^1(\IR^3)} \|\varphi\|_{C^1(\IR^3)}.       
	\end{align}
	Now, we take the sum in $ x $ and use that $\cup_x q_x = \IR^3$ up to a nullset. Furthermore, we observe that we 
	only have to take into account those cubes that lie in the support of $ \varphi $. The number of such cubes is 
	bounded by $ C d^{-3} = C \mu r^{-1} $. Therefore, combining the above estimates leads to
	\[
		|\langle \varphi , \sum_x e^{-\beta_1 |x|} M_x u - \mu J A^{\beta_1}u \rangle| \leq C \mu d.
	\]
	This proves the convergence for $ n = 1 $.
	
	For $ n = 2 $, we write
	\begin{equation}
		\label{eq:pSplitTT}
	\begin{aligned}
		&\sum_{x_1} \sum_{x_2 \neq x_1} e^{-\beta_1 |x_1|}e^{-\beta_2 |x_2|} M_{x_1} G_0 M_{x_2} u - 
		(\mu)^2 J A^{\beta_1} G_0 J A^{\beta_2} u \\
		&= \bigg(\sum_{x_1} e^{-\beta_1 |x_1|} M_{x_1} -  \mu JA^{\beta_1} \bigg)  \mu G_0 J A^{\beta_2}u \\
		{}  &+ \sum_{x_1} e^{-\beta_1 |x_1|} M_{x_1} G_0 
		\bigg(\sum_{x_2 \neq x_1} e^{-\beta_2 |x_2|}M_{x_2} -  \mu JA^{\beta_2}\bigg)u.
	\end{aligned}
	\end{equation}
	The first term converges to zero weakly in $H^{-1}(\IR^3)$ by the assertion for $n = 1$. 
	We observe that for all $x_2 \neq x_1 \in \Gamma_r$, and all $z \in B_{x_1}$,
	\begin{align}
		 &\bigg | \int_{\partial B_{x_2}} e^{-\beta_2 |x_2|} (u)_{x_2} \Phi (z-y) \dd z 
		-   \mu    \int_{q_{x_2}} e^{-\beta_2 |y|}u(y) \Phi(z-y) \dd y \bigg | \\
		 & \leq C r d e^{d-|x_2|}\|u\|_{C^1(\IR^3)} \|\Phi(z-\cdot)\|_{C^1(Q_{x_2)}} \\
		& \leq C r d \|u\|_{C^1(\IR^3)} e^{d-|x_2 - x_1|} 
		\left( \frac{1}{|x_2 - x_1|} + \frac{1}{|x_2 - x_1|^2} \right).
	\end{align}
	Taking the sum over $ x_2 \neq x_1 $ yields
	\begin{equation}
		\label{eq:pEstimateOtherBall}
		e^{-\beta_1 |x_1|}  \bigg | \bigg(G_0 \sum_{x_2 \neq x_1} \tilde{M}_{x_2} u -  \mu J A^{\beta_2} u \bigg)_{x_1} \bigg |
		\leq C \mu e^d \|u\|_{C^1(\IR^3)} d. 
	\end{equation}
	Note that it is crucial for deriving this bound that the sum runs only over $x_2$ different from $x_1$.
	Testing again by $ \varphi \in C_c^\infty(\IR^3) $, we conclude that also the second term in Equation 
	\eqref{eq:pSplitTT} tends to zero weakly in $ H^{-1}(\IR^3)$.
	
	Convergence of the higher order terms is proven by induction.
\end{proof}

%% file: Stokes.tex
\section{Adaptation to Stokes Equations}
\label{sec:SpacesAndOperators}

In this section, we will adapt the previous results for the Poisson equation to the case of Stokes equations.
We will not repeat everything from the previous sections but rather point out the necessary modifications.

Working only in spaces of divergence free functions, the presence of the pressure in the Stokes equations 
can be in principle ignored for the definition of all the operators needed. 

\begin{definition}
	We define $\dot{H}^1_\sigma(\IR^3;\IR^3) \subset \dot{H}^1(\IR^3;\IR^3) $ to be the closed subspace of divergence free functions, and
	$\dot{H}^{-1}_\sigma(\IR^3;\IR^3)$ its dual space.
\end{definition}

\begin{notation}
	To improve readability, we will from now on write $\dot{H}^1(\IR^3)$ instead of $\dot{H}^1(\IR^3;\IR^3)$ 
	and similarly for $\dot{H}^{-1}(\IR^3;\IR^3)$, $\dot{H}^1_\sigma(\IR^3;\IR^3)$, etc.
\end{notation}

\begin{remark}
Note that $\dot{H}^{-1}_\sigma(\IR^3) \subset \dot{H}^{-1}(\IR^3)$.
	Here, the inclusion for $f \in \dot{H}^{-1}_\sigma(\IR^3)$ to a function in $\dot{H}^{-1}(\IR^3)$ is given by
 $\langle u, f\rangle := \langle P_\sigma u, f \rangle$ for all $ u \in \Hdot(\IR^3)$, 
where $P_\sigma$ is the orthogonal projection from $\Hdot(\IR^3)$ to $\Hdot_\sigma(\IR^3)$.
\end{remark}

\begin{lemma}
	\label{lem:sto}
	Let $ f \in \dot{H}^{-1} (\IR^3) $. Then, the Stokes equations
	\begin{equation}
	\begin{aligned}
		-\Delta u & = -\nabla p + f, \\
		\dv u &= 0
	\end{aligned}
		\label{eq:sto}
	\end{equation}
	has a unique weak solution $ (u,p) \in \dot{H}^1_\sigma(\IR^3) \times L^2(\IR^3) $.
	The solution operator $ \bar{G}_0$ for the velocity field is given by
	\begin{equation}
		\bar{G}_0 f = \Phi \ast f,
	\end{equation}
	where
	\begin{equation}
		\Phi(x) := \frac{1}{8 \pi} \left( \frac{1}{|x|} + \frac{x \otimes x}{|x|^3} \right).
	\end{equation}
	Moreover, the restriction of the solution operator to $\dot{H}^{-1}_\sigma$, which we denote by $G_0$, is an isometric isomorphism.
\end{lemma}

\begin{lemma}
	\label{lem:stobdry}
	Let $\Omega \subset \IR^3$ be open.
	Then, for every $ f \in  \dot{H}^{-1}(\IR^3) $,	the problem
	\begin{equation}
	\begin{aligned}
		\label{eq:stokesBall}
		-\Delta u  &= - \nabla p + f \quad \text{in} ~ \IR^3 \backslash \overline{\Omega}, \\
		\dv u &= 0,\\
		u &= 0 \quad \text{in} ~ \overline{\Omega}, \\
		p &= 0 \quad \text{in} ~ \overline{\Omega}
	\end{aligned}
	\end{equation}
	has a unique weak solution $ (u,p) \in \dot{H}^1_\sigma(\IR^3) \times L^2(\IR^3)$.
	Moreover,
	\begin{equation}
		u = P_{\Omega} \bar{G}_0,
	\end{equation}
	where $ P_{\Omega} $ is the orthogonal projection from $ \dot{H}^1_\sigma (\IR^3) $ to the subspace
	$ \dot{H}^1_{0,\sigma}(\IR^3 \backslash \overline{\Omega})$.
\end{lemma}

\begin{remark}
	Analogous to $H^1_0(\IR^3 \backslash \overline{\Omega})$,  we use the convention
	\[
	\dot{H}^1_{0,\sigma}(\IR^3 \backslash \overline{\Omega}) :=
	 \{u \in \dot{H}^1_{\sigma}(\IR^3 \backslash \overline{\Omega}) \colon u=0 \text{ in } \Omega \}.
	\] 
\end{remark}

\begin{remark}
	The condition $p = 0 $ in $\overline{\Omega}$ in Equation \eqref{eq:stokesBall} ensures uniqueness. Indeed, dropping
	this condition, $p$ can be chosen equal to any constant in every bounded connected component of $\Omega$.
\end{remark}

	Again, for a particle $i$, we define the orthogonal projection $Q_i = 1 - P_i$, where $P_i = P_{B_i}$ and
	notice that
	\[
	 	\dot{H}^1_{0,\sigma}(\IR^3 \backslash \overline{B_i})^{\perp_\sigma} 
	 	= \{ u \in \Hdot_\sigma(\IR^3) \colon -\Delta u = -\nabla p \text{ in } \IR^3 \backslash \overline{B_i}
	 	\text{ for some } p \in L^2(\IR^3 \backslash  \overline{B_i})\},
	\]
	 where $\perp_\sigma$ indicates that we take the orthogonal complement with respect to $\Hdot_\sigma(\IR^3)$.

Notice that $G_0^{-1} Q_i u \in \dot{H}^{-1}_\sigma(\IR^3)$ is supported in $\overline{B_i}$, i.e., $\langle v,G_0^{-1} Q_x u\rangle = 0$ for every
$v$ in $ \Hdot_{0,\sigma}(\IR^3 \backslash B_x)$. This, however, does not mean
that $G_0^{-1} Q_i u$, viewed as an element of $\dot{H}^{-1}(\IR^3)$, is supported in $\overline{B_i}$. 

In the case of Poisson equation, we often used cutoff functions to exploit that a function $f \in \dot{H}^{-1}_\sigma(\IR^3)$ is supported in $\overline{B_i}$.
However, multiplication with a cutoff function destroys the property of a function to be divergence free.
Therefore, the following Lemma is needed.

\begin{lemma}
	\label{lem:pressureEstimate}
	Assume $ f \in \dot{H}^{-1}_\sigma(\IR^3) $ is supported in 
	$ \overline{B_i} $. Then, there exists a unique $p \in L^2(\IR^3)$ with $p=0$ in $B_i$ such that 
	$ \tilde{f} := f +\nabla p $ is supported in $\overline{B_i}$ as a function in $\dot{H}^{-1}(\IR^3)$.
	Moreover, 
	 $\|\tilde{f}\|_{\dot{H}^{-1}(\IR^3)} \leq C \|f\|_{\dot{H}^{-1}(\IR^3)}$
	for a universal constant $ C $.
	We denote by $S$ the operator that maps $f$ to $\tilde{f}$.
	
\end{lemma}

\begin{proof}
	Since $ f \in \dot{H}^{-1}_\sigma(\IR^3) $ is supported in 
	$ \overline{B_i} $, we have $\langle f,v\rangle = 0$ for all 
	$v \in \Hdot_{0,\sigma}(\IR^3 \backslash B_i)$. Hence, there exists a unique $p \in L^2(\IR^3\backslash \overline{B_i})$ 
	such that $f = -\nabla p$ in $\IR^3 \backslash \overline{B_i}$ and we can set $p=0$ 
	in $B_i$. By Lemma \ref{lem:divergenceSolution} below, we can find $u \in \Hdot_0(\IR^3 \backslash B_i)$ such that $\dv u = p$ and 
	$\|u\|_{\Hdot(\IR^3)} \leq C \|p\|_{L^2(\IR^3)}$. Hence, 
	\[
	  \|u\|_{\Hdot(\IR^3)} \|f\|_{\dot{H}^{-1}(\IR^3)}
	   \geq \langle u, f \rangle   = \langle u, - \nabla p \rangle	= \|p\|^2_{L^2(\IR^3)},
	\]  
	and thus $\|p\|_{L^2(\IR^3)} \leq C \|f\|_{\dot{H}^{-1}(\IR^3)}$.
	Hence, $\tilde{f} := f + \nabla p$ is supported in $\overline{B_i}$ as a function in $\dot{H}^{-1}(\IR^3)$, and
	 $\|\tilde{f}\|_{\dot{H}^{-1}(\IR^3)} \leq C \|f\|_{\dot{H}^{-1}(\IR^3)}$.
\end{proof}

The following Lemma can be found in every standard textbook on Stokes equations, e.g., in \cite{Ga11}.

\begin{lemma}
	\label{lem:divergenceSolution}
	Let $\Omega \subset \IR^3$ be a locally lipschitzian bounded or exterior domain. Then there exists a constant $C$
	with the following property.
	For all $f \in L^2(\Omega)$, that satisfies
	\[
		\int_\Omega f \dd x = 0,
	\] 
	if $\Omega$ is a bounded domain, there exists $u \in H^1_0(\Omega)$ such that 
	\[
		\dv u = f
	\]
	and
	\[	
		\|\nabla u\|_{L^2(\Omega)} \leq  C \|f\|_{L^2(\Omega)}.
	\]
\end{lemma}
\begin{remark}
	The constant $C$ is invariant under scaling of $\Omega$.
\end{remark}

Now one can define the operator $L$ analogously to
the corresponding operator for the Poisson equation from Definition \ref{def:pL}.
Using Lemma \ref{lem:pressureEstimate}, the estimate for $L$ (cf. Lemma \ref{lem:Lbounded}) follows in the same manner as before.
Then, Theorem \ref{StokesConvergence} follows immediately from Proposition \ref{pro:nonuniformAbstractProjection} and Lemma \ref{lem:stobdry}.

\subsection{Homogenization}

Corresponding to Definition \ref{def:monopole}, we introduce the following operators.

\begin{definition}	
	We define $T_x : \Hdot_\sigma(\IR^3) \to \dot{H}^{-1}(\IR^3)$ by
	$T_x = S G_0^{-1} Q_x$, where $S$ is the operator from Lemma \ref{lem:pressureEstimate}.

	Moreover, we define the uniform force density approximation of the operator $T$ 
	to be the operator $M_x : \Hdot_\sigma(\IR^3) \to \dot{H}^{-1}(\IR^3)$,
	\[
		(M_x u)(y) = \frac{3(u)_{x,r}}{2r}
	\]
	\end{definition}

Note that the definition of $M_x$ differs from the corresponding operator for the Poisson equation (cf. Definition \ref{def:monopole}) by a factor
$3/2$. The reason for this is that the electrostatic capacity of a ball of radius $r$ is 
$4 \pi r$. The corresponding quantity for the Stokes equations, however, is the absolute value of the Stokes' drag force acting on a ball of radius $r$ moving with unit speed in a fluid which is  at rest at infinity, which is $6 \pi r$.

Lemma \ref{lem:extest} used in the proof Lemma \ref{lem:pMonopapprox} has to be replaced by the following Lemma.

\begin{lemma}
	\label{lem:Stokesextest}
	For $ r>0 $ and $ x \in \IR^3$,  let $ H_r := \left\{ u \in  H^1(B_r(x)) \colon \int_{B_r(x)} u = 0 \right\} $. 
	Then for all $ r >0 $	there exists an extension operator $ E_r \colon H_r \to H^1_0(B_{2r}(x)) $ such that
	\begin{equation}
		\label{eq:Stokesextest}
		\| \nabla E_r u \|_{L^2(B_{2r}(x))} \leq C \| \nabla u \|_{L^2(B_r(x))} \qquad \text{for all} \quad u \in H_r,
	\end{equation}
	where the constant $ C $ is independent of $ r $.
\end{lemma}

\begin{proof}
	For $ r = 1 $ let $ E_1 \colon H^1(B_1(x)) \to H^1_0(B_2(x)) $ be a continuous extension operator.
	Then, by the Poincaré inequality on $ H_1 $ we get for all $ u \in H_1 $
	\begin{equation}
		\| \nabla E_1u \|_{L^2(B_{2}(x))} \leq \| E_1u \|_{H^1(B_{2}(x))} \leq C \| u \|_{H^1(B_1(x))}
		\leq C \| \nabla u \|_{L^2(B_1(x))}
	\label{eq:Stokesextest1}
	\end{equation}
	The assertion for general $ r > 0 $ follows from scaling by defining 
	$ (E_r)u(x) := (E_1u_r)({\frac{x}{r}}) $ where $ u_s(x) := u(sx) $.
\end{proof}

These are the only things that change in the proof of the homogenization result, Theorem \ref{HomogStokes},
except for the result about locally uniform convergence in the particle configuration.
For the Poisson equation, this result was stated in Proposition \ref{pro:pSolByScattering}.
The analogous statement for the Stokes equations remains valid.

However, the proof of Lemma \ref{lem:LCoerciveInCompacta} and \ref{lem:decayQuasiHarmonic} needed in the proof of Proposition
\ref{pro:pSolByScattering} have to be modified due to the use of cutoff functions.
Corresponding to  Lemma \ref{lem:LCoerciveInCompacta} and \ref{lem:decayQuasiHarmonic}, we will prove
 Lemma \ref{lem:SLCoerciveInCompacta} and \ref{lem:SpolynomialDecay} .
For the proof Lemma \ref{lem:SLCoerciveInCompacta}, we need the following lemma.

\begin{lemma}
	\label{cor:solenoidalExtension}
	Let $\Omega \subset \IR^3$ be a bounded and locally lipschitzian domain and assume $v \in H^1(\Omega)$ satisfies
	\[
		\int_\Omega v \cdot \nu = 0.
	\]
	Then, for any $R>0$ and $x \in \IR^3$ such that $\Omega \subset \subset B_R(x)$, there exists $u \in H^1_0(\Omega)$ such that 
	\begin{align}
		u &= v \text{ in } \Omega \\
		\dv u &= 0 \text{ in } B_R(x) \backslash \overline{\Omega}
	\end{align}
	and
	\[	
		\| u\|_{H^1(B_R)} \leq  C \|v\|_{H^1(\Omega)},
	\]
	where the constant depends only on the domains $\Omega$ and $B_R(x)$.
	
	In particular, for any $v \in H^1(B_r)$ with $\int_{b_r} v \cdot \nu = 0$, we can find $u \in H^1_0(B_{2r}(x))$ such that 
	\begin{align}
		u &= v \text{ in } B_r(x) \\
		\dv u &= 0 \text{ in } B_{2r}(x) \backslash B_r(x)
	\end{align}
	and
	\[	
		\|\nabla u\|^2_{H^1(B_{2r}(x))} \leq  \frac{C}{r^2} \|v\|^2_{L^2(B_r(x))} + C\|\nabla v\|^2_{L^2(B_r(x))} 
		\leq C \|\nabla v\|^2_{L^2(\IR^3)},
	\]
	where the constant is independent of $r$ and $v$.
\end{lemma}

\begin{proof}
	We take any (not necessarily divergence free) extension $u_1$ of $v$ to $B_R(x)$ that satisfies the estimate, 
	and take a solution $u_2 \in H^1(B_R \backslash \overline{\Omega})$ of $\dv u_2 = -\dv u_1$
	provided by Lemma \ref{lem:divergenceSolution} and define $ u = u_1 + u_2$.
	
	The second assertion follows from scaling, and the last inequality is a consequence of H\"older's inequality and the 
	Gagliardo-Nirenberg-Sobolev inequality.
\end{proof}

\begin{lemma}
	\label{lem:SLCoerciveInCompacta}
	Let 	$ u \in \dot{H}_{0,\sigma}^1(\IR^3 \backslash K_r)^{\perp_\sigma} $ and $R > 0$. We define $v \in \dot{H}^1_\sigma(\IR^3)$ to be the solution
	to
	\begin{align}
		-\Delta v &= -\nabla p \quad \text{in} ~ \IR^3 \backslash (K_r \cap \overline{B_R(0)}), \\
		\dv v &= 0, \\
		v &= u  \quad \text{in} ~ K_r \cap \overline{B_R(0)}.
	\end{align}
	Then,
	\[
		 (L_r u,u)_{\dot{H}^1(\IR^3)} \geq c e^{-R} \| v\|_{\dot{H}^1(\IR^3)}^2,
	\]
	where $c>0$ is a universal constant.
\end{lemma}

\begin{proof}
	By the variational form of the equation for $ v $, we know that $v$ is the function of minimal norm 
	in the set $ X_v := \{ w \in \dot{H}^1_\sigma(\IR^3) \colon w = v ~ \text{in} ~ K \cap \overline{B_R} \}$. 
	For every $x$ in $\Gamma_r \cap B_{R=r}$, Corollary \ref{cor:solenoidalExtension} provides functions 
	$v_x \in H^1_0(B_{2r}(x)$ with $\|v_x\|_{\Hdot(\IR^3)} \leq C \|Q_x v\|_{\Hdot(\IR^3)}$ such that $v_x = Q_x v = v $ in $B_x$.
	 Clearly, $\sum_{x \in B_{R+r}} v_x \in X_u$, and hence,
	\begin{align}
		\langle Lv , v \rangle &= \sum_x e^{-|x|} \|Q_x v \|_{\dot{H}^1(\IR^3)}^2 \\
		&\geq c e^{-R} \sum_{x \in B_{R+r}} \|v_x \|_{\dot{H}^1(\IR^3)}^2 \\
		&= c e^{-R} \|\sum_{x \in B_{R+r}} v_x \|_{\dot{H}^1(\IR^3)}^2 \\
		&\geq c e^{-R} \|v\|^2_{\dot{H}^1(\IR^3)}. \qedhere
	\end{align}
\end{proof}

For the proof Lemma \ref{lem:SpolynomialDecay}, we need the following lemma.

\begin{lemma}
	\label{lem:projest}
	Let $ u \in H^1(\IR^3) $ and $ x \in \IR^3 $. Assume $ 0<\rho<R$. Then
	\begin{equation}
		\| u \|^2_{L^2(B_{\rho}(x))} 
		\leq C \left( \frac{\rho^3}{R^3} \|u\|_{L^2(B_R(x))}^2 + \rho^2 \|\nabla u\|_{L^2(B_R(x))}^2\right),
	\end{equation}
	where $ C$ is a universal constant.
	
	In particular, for all particle configurations with capacity $\mu$ and  all $u \in H^1(\IR^3)$, we have
		\begin{equation}
		\| u \|^2_{L^2(K_r)} \leq C \mu \|u\|_{L^2(\IR^3)}^2 + C \mu \|\nabla u\|_{L^2(\IR^3)}^2.
	\end{equation}
\end{lemma}

\begin{proof}
	Define $ (u)_{R,x} = \fint_{B_{R}(x)} u $. Then, using Lemma 	\ref{lem:extest}	we get
	\begin{equation}
		\begin{aligned}
		\|u-(u)_{R,x}\|_{L^2(B_{\rho}(x))} 
		&\leq \|u-(u)_{R,x}\|_{L^6(B_{\rho}(x))} \|1\|_{L^3(B_{\rho}(x))} \\
		&\leq C \rho \|\nabla E_{R}(u-(u)_{R,x})\|_{L^2(B_d(x))} 
		\leq C \rho \|\nabla u \|_{L^2(B_{R}(x))}.
		\end{aligned}
	\end{equation}
	Furthermore,
	\begin{equation}
		\| (u)_{R,x} \|^2_{L^2(B_{\rho}(x))} = C \rho^3 \left(\fint_{B_{R}(x)} u \dd x \right)^2 
		\leq C \rho^3 \fint_{B_{R}(x)} u^2 \dd x= C \frac{\rho^3}{R^3}\| u \|^2_{L^2(B_{R}(x))}.
	\end{equation}
	Combining these two estimates yields the assertion.
\end{proof}

\begin{lemma}
	\label{lem:SpolynomialDecay}
	For all $\mu > 0$ and $\rho>0$, there exists a nonincreasing function 
	$ e_{\mu,\rho} \colon \IR_+ \to \IR_+$ with $ \lim_{s \to \infty} e_{\mu,\rho}(s) = 0$ that has the following property.
	All $w \in \dot{H}_0^1(\IR^3 \backslash K_r)^\perp $  with $w = 0$ in 
	$ K_r \cap B_R(0)$ satisfy
	\[
		\| \nabla w\|_{L^2(B_\rho(0))} \leq e_{\mu,\rho}(R) \| \nabla w \|_{L^2(\IR^3)},
	\]
	for all $R \geq \rho$
	if $r$ is sufficiently small.
\end{lemma}

\begin{proof}
		Fix a particle configuration with capacity $\mu $ and $d<1/(2\sqrt{3})$, and fix $R$, $\rho$, and $w$ according to the
	assumptions.
	Assume $s \geq 1$ satisfies $ 2s \leq R $. Note that $w$ is the function of minimal norm in the set
	\[
		X_w := \{ v \in \Hdot_\sigma \colon v=0 \text{ in } K_r \cap B_{2s}(0), ~ v = u \text{ on } \partial B_{2s}(0)\} 
	\] 
	Define  $\eta \in C^1(\IR^3) $ to be a cut-off function with $\eta = 1 $ in $\IR^3 \backslash B_{2s(1-3r)}(0)$, $eta = 0 $ in $B_{s(1+3r)}$,
	and $|\nabla \eta | \leq C/s$. Then, $v_1:= \eta w$ has the right boundary condition to be in the set $X_w$ but fails to be divergence free.
	Indeed, $\dv v_1 = \nabla \eta \cdot w$.
	Therefore, we use Lemma \ref{lem:divergenceSolution} to find a function $v_2 \in \Hdot_0 (B_{2s} \backslash B_s)$ with
	$\dv v_2 = - \dv v_1$ and 
	\[
		 \| \nabla v_2 \|_{L^2(B_{2s} \backslash B_s)} \leq C  \| \dv v_1 \|_{L^2(B_{2s} \backslash B_s)} 
		 \leq \frac{C}{s} \| w \|_{L^2(B_{2s} \backslash B_s)}.
	\]
	Now $v_1 + v_2$ is divergence free and equals $w$ on $\partial B_{2s}$. To match the boundary conditions in $K_r \cap B_{2s}(0)$,
	we use Corollary \ref{cor:solenoidalExtension}. For $x \in \Gamma_r \cap (B_{2s(1-2r)} \backslash B_{s(1+2r)}) $
	it provides a function $v_x \in H^1_{0,\sigma}(B_{2r}(x))$  with
	$v_x = - v_2$ in $B_x$ and 
	\begin{align}
		\| v_x \|^2_{\Hdot} & \leq \frac{C}{r^2} \| v_2 \|^2_{L^2(B_x)} + C \| \nabla v_2 \|^2_{L^2(B_x)}  \\
		& \leq C ( \mu \| v_2 \|^2_{L^2(B_\frac{d}{2}(x)} + \| \nabla v_2 \|^2_{L^2(B_\frac{d}{2}(x)} \\
		&\leq C(s^2 \mu + 1) \| \nabla v_2 \|^2_{L^2(B_\frac{d}{2}(x)},
	\end{align}
	where we used Lemma \ref{lem:projest} for the second estimate and the Poincaré inequality in $ \Hdot_0 (B_{2s} \backslash B_s)$ 
	for the last one.
	By construction, $v:= v_1 + v_2 + \sum_{x \in \Gamma_r\cap (B_{2s(1-2r)} \backslash B_{s(1+2r)})} v_x $ is an element of $X_w$.
	Therefore,
	\begin{align}
		0 &\leq \| \nabla v \|_{L^2(\IR^3)}^2 - \| \nabla w \|_{L^2(\IR^3)}^2 \\
		& \leq C\|\nabla w \|^2_{L^2(B_{2s} \backslash B_s)} + C(\frac{1}{s^2} + \mu) \| w \|^2_{L^2(B_{2s} \backslash B_s)} 
		- \|\nabla w\|^2_{L^2(B_s)}.
 	\end{align}
	Since $s \geq 1 $ by assumption, the factor $s^{-2}$ can be dropped.
	Using the Poincaré inequality in the annulus $B_{2s} \backslash B_s $, provided by
	Lemma \ref{lem:poincareAnnulus} below, we deduce
	\[
		\|\nabla w\|^2_{L^2(B_s)} \leq C (1+\mu^{-1}) \| \nabla w \|^2_{L^2(B_{2s} \backslash B_s)}.
	\]
	Using again the hole filling technique and iterating from $s:= \max\{\rho,1\}$ until $2^k s \geq R/2 $ concludes the proof. 
\end{proof}

%% file: Conclusion.tex
\section{Conclusions}

We have analyzed the convergence properties of the Method of Reflections
for both Poisson and Stokes equations with Dirichlet boundary conditions.
For typical particle configurations extending to the whole space, convergence does not hold.
However, a modified method has been obtained that ensures convergence 
 for particle configurations with bounded capacity density.
 
Using this method, we have proven classical homogenization results in unbounded domains for regular
particle configurations and sources $f \in H^{-1}(\IR^3)$. For the Poisson equation it
was proven in \cite{NV1}, \cite{NV} that this result can be extended 
to sources $f\in L^{\infty}\left(
\mathbb{R}^{n}\right)  $. The proof in \cite{NV1}, \cite{NV} relies heavily in the
derivation of the so-called screening estimate, which states that the
fundamental solution for the Laplace equation in a perforated domain with
Dirichlet boundary conditions decays exponentially. 
In \cite{NV1}, \cite{NV}, this decay was proven by means of the Maximum Principle.
As we have seen in Lemma \ref{lem:decayQuasiHarmonic} (cf. also Remark \ref{rem:exponentialDecay}), 
it is possible to derive such exponential screening estimates without using Maximum Principles,
relying instead on Poincaré estimates for perforated domains (cf. Corollary \ref{cor:PoincarePerforated}
and Lemma \ref{lem:poincareAnnulus}). Therefore, the results can hoped to be extended to more general
elliptic operators.

For the Stokes equations, however, we do not have such a strong decay estimate (cf. Lemma \ref{lem:SpolynomialDecay}).
In fact, the solutions of the Brinkman equations \eqref{S2E9} with compactly supported sources $g \in L^\infty(\IR^3)$ 
decay only cubic in the distance to the support of $g$ (cf. \cite{AHF}).
Therefore, the solutions with sources $f \in L^\infty(\IR^3)$ cannot be expected to be bounded.
 
The boundary conditions used in  \cite{Luke} are
not the homogeneous Dirichlet conditions but the set of natural boundary
conditions for sedimenting particles (or an analogous set of Neumann-Dirichlet
boundary conditions in the case of the Poisson equation). It is worth to
indicate that the screening effects, which have been discussed above, can be
expected to be rather different for the set of boundary conditions in
\cite{Luke} and for Dirichlet boundary conditions that we considered. Using again the
electrostatic analogy, the Dirichlet boundary conditions we consider in this paper
are those corresponding to grounded conducting particles, while those in \cite{Luke}
correspond to isolated conducting particles.  Hence, the Dirichlet boundary conditions result
in the onset of induced charges at the particles which are proportional to the
capacity of the particles. On the contrary, the type of boundary conditions
used in \cite{Luke} results in the onset of induced dipoles at the particles,
instead of charges. The potentials produced by dipoles decay faster than the
ones produced by charges and as a consequence screening effects and collective
particle interactions might be expected to be less relevant.
 Understanding this type of dipole induced screening effects is an interesting issue,
which deserves further investigation.